\documentclass[leqn,11pt]{amsart}
\usepackage{a4}
\usepackage{amsmath,amssymb,amscd,bbm}
\usepackage{enumerate}
\usepackage[mathscr]{eucal}
\usepackage[all]{xy}
\usepackage{color}
\usepackage[T1]{fontenc}
\newtheorem{thm}{Theorem}[section]
\newtheorem{lemma}[thm]{Lemma}

\newtheorem{cor}[thm]{Corollary}

\theoremstyle{definition}

\newtheorem{exa}[thm]{Example}

\newtheorem{rem}[thm]{Remark}
\newtheorem{rems}[thm]{Remarks}
\numberwithin{equation}{section}

\newcommand{\vanish}[1]{\relax}       % comments out inclosed text
\newcommand{\emdf}{\bf}               % Definitions in Text
\def\qedsymbol{\hbox to 1ex{\llap{\rule{0.25pt}{1ex}}\rlap{\rule{1ex}{0.25pt}}\lower0.25pt\rlap{\raise1ex\rlap{\rule{1ex}{0.25pt}}}\hskip1ex\llap{\rule{0.25pt}{1ex}}}}
\def\rlqed{\rlap{\rule{\hsize}{0pt}\kern-1ex\kern-1em\qed}} %% \qed rechts

\newcounter{aufzi}
\newenvironment{aufzi}{\begin{list}{ {\upshape\alph{aufzi})}}{
        \usecounter{aufzi}
        \topsep1ex
        \parsep0cm
        \itemsep0.8ex
        \leftmargin1cm
        \labelwidth0.5cm
        \labelsep0.3cm
}}
{\end{list}}

\newcounter{aufzii}
\newenvironment{aufzii}{\begin{list}{\hfill {\upshape
(\roman{aufzii})}}{
        \usecounter{aufzii}
        \topsep1ex
        \parsep0cm
        \itemsep1ex
        \leftmargin1cm
        \labelwidth0.5cm
        \labelsep0.3cm
}}
{\end{list}}

\newcounter{aufziii}
\newenvironment{aufziii}{\begin{list}{ {\upshape\arabic{aufziii})}}{
        \usecounter{aufziii}
        \topsep1ex
        \parsep0cm
        \itemsep0.8ex
        \leftmargin1cm
        \labelwidth0.5cm
        \labelsep0.3cm
}}
{\end{list}}



\newcommand{\calA}{\mathcal{A}}

\def\bfK{\mathbf{K}}

\def\bfU{\mathbf{U}}

\def\bfX{\mathbf{X}}		
\def\bfY{\mathbf{Y}}
\def\bfZ{\mathbf{Z}}

\def\ue{\mathrm{e}}

\def\ui{\mathrm{i}}

\newcommand{\usigma}{\sigma}  % {\allmodesymb{\greeksym}{s}}
\newcommand{\upi}{\pi}      %{\allmodesymb{\greeksym}{p}}
\renewcommand{\ue}{\mathrm{e}}    %  for Euler's number
\renewcommand{\ui}{\mathrm{i}}    %  for imaginary unit
\def\id{\mathop{\mathrm{id}}\nolimits}   %for the identity map on a set
\newcommand{\car}{\mathbf{1}}           % constant function 1
\newcommand{\T}{\mathbb{T}}
\newcommand{\R}{\mathbb{R}}     % real numbers etc
\newcommand{\N}{\mathbb{N}}
\newcommand{\Z}{\mathbb{Z}}

\newcommand{\C}{\mathbb{C}}

\newcommand{\torus}{\T}          % torus

\newcommand{\st}{\,:\,}     % "suchthat"

\newcommand{\without}{\setminus}

\newcommand{\Bigcap}[2][\relax]{%
 \ifx#1\relax \bigcap_{#2}
 \else \bigcap^{#1}_{#2}
 \fi}
\newcommand{\Bigcup}[2][\relax]{%
 \ifx#1\relax \bigcup_{#2}
 \else \bigcup^{#1}_{#2}
 \fi}

\newcommand{\after}{\circ}
\newcommand{\pfeil}{\longrightarrow}

\DeclareMathOperator{\ran}{ran}
\DeclareMathOperator{\fix}{fix}
\newcommand{\restrict}[1]{|_{#1}}
\def\res{\restrict} % Einschr\"ankung
\newcommand{\conj}[1]{\overline{#1}}   % conjugation
\newcommand{\abs}[1]{\left\vert#1\right\vert}   % modulus, if not too high
\newcommand{\cls}[1]{\overline{#1}}     % closure
\newcommand{\tensor}{\otimes}

\DeclareMathOperator{\End}{End}
\DeclareMathOperator{\Aut}{Aut}

\def\haar{\mathrm{m}}   % Haar measure
\renewcommand{\mod}[1]{\,\,\,\,(\mathrm{mod}\,#1)\,}     % modulo
\def\fact#1#2{#1/#2}
\def\tfact#1#2{#1/#2}

\def\fact#1#2{{\raise0.2em\hbox{$#1$}\kern-0.2em/\kern-0.1em\lower0.2em\hbox{$#2$}}}
\def\tfact#1#2{{\raise0.1em\hbox{\small$#1$}\kern-0.1em/\kern-0.1em\lower0.1em\hbox{\small$#2$}}}

\newcommand{\norm}[2][\relax]{%                             % norm
   \ifx#1\relax \ensuremath{\left\Vert#2\right\Vert}
   \else \ensuremath{\left\Vert#2\right\Vert_{#1}}
   \fi}

\newcommand{\Bnorm}[2][\relax]{%        % Big norm, but usually no too big
   \ifx#1\relax \ensuremath{\Bigl\Vert#2\Bigr\Vert}
   \else \ensuremath{\Bigl\Vert#2\Bigr\Vert_{#1}}
   \fi}
\def\Id{\mathrm{I}}           %Identity operator
\newcommand{\Pspec}{\usigma_{\mathrm{p}}}

\DeclareMathOperator{\spann}{span}     % linear span
\def\cspan{\overline{\spann}}           % closed linear span
\def\cl{\mathrm{cl}}                   % closure
\makeatletter
\newcommand{\tdprod}[2]{\ensuremath{%
  \setbox0=\hbox{\ensuremath{\langle#1,#2 \rangle}}
  \dimen@\ht0
  \advance\dimen@ by \dp0 (#1\rule[-\dp0]{0pt}{\dimen@}\,|#2\hspace{1pt})}}
\newcommand{\dprod}[2]{\ensuremath{%
  \setbox0=\hbox{\ensuremath{\left\langle#1,#2\right\rangle}}
  \dimen@\ht0
  \advance\dimen@ by \dp0 \left\langle\left.#1\rule[-\dp0]{0pt}{\dimen@}\,\right|#2\hspace{1pt}\right\rangle}}

\newcommand{\bdprod}[2]{\ensuremath{%
  \setbox0=\hbox{\ensuremath{\bigl\langle#1,#2\bigr\rangle}}
  \dimen@\ht0
  \advance\dimen@ by \dp0 \bigl\langle#1\bigl|\rule[-\dp0]{0pt}{\dimen@}\bigr.#2\hspace{1pt}\bigr\rangle}}
\newcommand{\Bdprod}[2]{\ensuremath{%
  \setbox0=\hbox{\ensuremath{\Bigl\langle#1,#2\Bigr\rangle}}
  \dimen@\ht0
  \advance\dimen@ by \dp0 \Bigl\langle#1\Bigl|\rule[-\dp0]{0pt}{\dimen@}\Bigr.#2\hspace{1pt}\Bigr\rangle}}
\newcommand{\tsprod}[2]{\ensuremath{%
  \setbox0=\hbox{\ensuremath{(#1,#2)}}
  \dimen@\ht0
  \advance\dimen@ by \dp0 (#1\rule[-\dp0]{0pt}{\dimen@}\,|#2\hspace{1pt})}}
\newcommand{\sprod}[2]{\ensuremath{%
  \setbox0=\hbox{\ensuremath{\left(#1,#2\right)}}
  \dimen@\ht0
  \advance\dimen@ by \dp0 \left(\left.#1\rule[-\dp0]{0pt}{\dimen@}\,\right|#2\hspace{1pt}\right)}}
\newcommand{\bsprod}[2]{\ensuremath{%
  \setbox0=\hbox{\ensuremath{\bigl(#1,#2\bigr)}}
  \dimen@\ht0
  \advance\dimen@ by \dp0 \bigl(#1\bigl|\rule[-\dp0]{0pt}{\dimen@}\bigr.#2\hspace{1pt}\bigr)}}
\newcommand{\Bsprod}[2]{\ensuremath{%
  \setbox0=\hbox{\ensuremath{\Bigl(#1,#2\Bigr)}}
  \dimen@\ht0
  \advance\dimen@ by \dp0 \Bigl(#1\Bigl|\rule[-\dp0]{0pt}{\dimen@}\Bigr.#2\hspace{1pt}\Bigr)}}
\makeatother

\newcommand{\Ce}{\mathrm{C}}

\newcommand{\Ell}[2][\relax]{%     % Ell-p spaces
   \ifx#1\relax \mathrm{L}^{\mathrm{#2}}
   \else \mathrm{L}^{\mathrm{#2}}_{\mathrm{#1}}
   \fi}
\renewcommand{\Ell}[2][\relax]{%    ??????????????????????????????????????????
   \ifx#1\relax \mathrm{L}^{\!#2}
   \else \mathrm{L}^{\!#2}_{\mathrm{#1}}
   \fi}

\newcommand{\Wee}[2][\relax]{%    %   Sobolev spaces
   \ifx#1\relax \mathrm{W}^{\mathrm{#2}}
   \else \mathrm{W}^{\mathrm{#2}}_{\mathrm{#1}}
   \fi}
\newcommand{\Har}[2][\relax]{%       %Hardy spaces
   \ifx#1\relax \mathsf{H}^{\mathsf{#2}}
   \else   \mathsf{H}^{\mathsf{#2}}_{\mathrm{#1}}
   \fi}
\newcommand\Cex[2]{\mathbb{E}(#1\, |\, #2)}   % Conditional probability

\def\prX{\mathrm X}    % probability space
\def\prY{\mathrm Y}

\def\prZ{\mathrm Z}

\def\MDSX{(\prX;\varphi)}   % dynamical system,

\def\rlqed{\rlap{\rule{\hsize}{0pt}\kern-1ex\kern-1em\qed}}

\makeatletter
\def\maketag@@@@@#1{\llap{\hbox to\hsize{\m@th\normalfont#1}}%
\gdef\tagform@##1{\maketag@@@{(\ignorespaces##1\unskip\@@italiccorr)}}}

\def\eqtext#1{\gdef\tagform@##1{\maketag@@@@@{\ignorespaces##1\unskip\@@italiccorr\hfill}}\tag{#1}}%
\def\reqtext#1{\gdef\tagform@##1{\maketag@@@@@{\hfill\ignorespaces##1\unskip\@@italiccorr}}\tag{#1}}%
\def\leqtext#1{\gdef\tagform@##1{\maketag@@@@@{\ignorespaces##1\unskip\@@italiccorr}}\tag{#1}}%
\makeatother

\newcommand{\signature}{signature}
\newcommand{\vphi}{\varphi}
\newcommand{\prfskip}{\smallskip\noindent}
\newcommand{\sig}{\mathrm{sig}}

\newcommand{\ord}{\mathrm{ord}}

\date{\today}

\begin{document}

\title{On systems with  quasi-discrete spectrum}

\author{Markus Haase}
\address{Mathematisches Seminar, Christian-Albrechts-Universit\"at zu Kiel, 
Ludewig-Meyn-Str. 4, 24118 Kiel, Germany}
\email{haase@math.uni-kiel.de}

\author[Nikita Moriakov]{Nikita Moriakov}
\address{Delft Institute of Applied Mathematics, Delft University of Technology,
P.O. Box 5031, 2600 GA Delft, The Netherlands}

\email{n.v.moriakov@tudelft.nl}

\renewcommand{\subjclassname}{\textup{2000} Mathematics Subject
    Classification}

\date{1 June 2017}

\begin{abstract}
In this paper we re-examine the theory of systems with quasi-discrete
spectrum initiated in the 1960's by Abramov, Hahn, and Parry. 
In the first part, we give a simpler proof of the Hahn--Parry 
theorem stating that each minimal topological system with quasi-discrete
spectrum is isomorphic to
a certain  affine automorphism system on some compact Abelian group.  
Next, we show that a suitable application of Gelfand's theorem
renders Abramov's theorem --- the analogue of the Hahn-Parry theorem
for measure-preserving systems --- a straightforward corollary of
the Hahn-Parry result.

In the second part, independent of the first,  we present a 
 shortened proof of the fact that
each factor of a totally ergodic system with quasi-discrete spectrum 
(a ``QDS-system'') has
again quasi-discrete spectrum and that such systems have zero entropy. 
Moreover, we obtain a complete algebraic classification of the factors
of a QDS-system.

In the third part, we apply the results of the second to the (still open)
question whether a Markov quasi-factor of a QDS-system is already
 a factor of it. We show that this is true when the system
satisfies some algebraic constraint on the group of quasi-eigenvalues,
which is satisfied,  e.g., in the case of the skew shift.% 
\end{abstract}



\maketitle

\section{Introduction}\label{s.intro}

A classical problem in ergodic theory is to determine whether given
(measure-preserving) dynamical systems are isomorphic or not, 
to determine complete sets of isomorphism invariants at least for 
some classes of dynamical systems and, possibly, 
to find canonical representatives for the corresponding isomorphism classes.
 
The oldest result of this type is the Halmos--von Neumann theorem
from \cite{HalmosVonNeumann1942}, 
which says that the systems with \emph{discrete spectrum} 
are isomorphic to compact Abelian group rotations, and the isomorphism 
class is completely determined by the point spectrum of the associated
Koopman operator. 
%Here we say that system $\bfX$ has discrete spectrum if $\Ell{2}(\prX) = \clin\{ f \in \Ell{2}(\prX): f \text{ is an eigenvector of } T \}$.

In \cite{Abramov1962}, the notion of a system with discrete spectrum 
was generalized by Abramov 
to (totally ergodic) systems with  \emph{quasi-discrete spectrum}.
In  analogy to the results of Halmos--von Neumann,
Abramov could show that also this class has
a complete isomorphism invariant (the ``signature'', in our terminology) 
and canonical representatives.
Parallel to the original arguments of Halmos and von Neumann, Abramov
first established a ``theorem of uniqueness'' telling that 
two systems with quasi-discrete spectrum with same signature
are isomorphic, and then, in a ``theorem of existence'', 
showed that to each signature there is
a special system --- an affine automorphism on a compact Abelian group ---
with quasi-discrete spectrum and the given signature. As a combination
of these two results, he then obtained the main ``representation theorem''
that each totally ergodic system 
with quasi-discrete spectrum is isomorphic to 
an affine automorphism on a compact Abelian group.

A couple of years later, Hahn and Parry \cite{HahnParry1965} developed the
corresponding theory for topological dynamical systems. Their approach
was completely analogous: first to prove an isomorphism theorem, then 
a realization result; finally, as a corollary, the  representation
theorem. The results of Abramov and Hahn--Parry were incorporated by
Brown into   Chapter III of his classic book \cite{Brown1976}.
Although Brown's presentation is more systematic,
he  essentially copied Abramov's proof of the isomorphism theorem.

One purpose of this article is to introduce a considerable
simplification in the presentation of these results. 
In particular, we give direct proofs of the representation theorems
to the effect that the isomorphism theorems become corollaries.
We shall show, moreover, that 
the measure-preserving case 
is actually an immediate consequence of the topological case by virtue of
a good choice of a topological model via  Gelfand's theorem. 
(This underlines a general philosophy, already prominently 
demonstrated in the proof of the Halmos--von Neumann theorem in 
\cite[Chapter 17]{EFHN}.) 

Note that, in this approach, the realization results 
(``theorems of existence'') are not
needed any more neither for proving the representation nor the isomorphism
theorem. Nevertheless, the realization theorems are completing the picture,
and we include their proofs for the convenience of the reader.

\medskip

In the second part (Section 5), 
we present a purely operator theoretic proof of a (generalization of
a) result
of Hahn and Parry from \cite{HahnParry1968} which implies
among other things  that a factor of a totally ergodic
system with quasi-discrete spectrum has again quasi-discrete spectrum. 
Using our notion of ``signature'' we also give a complete
algebraic classification of the factors of such a system.
These results are completely 
independent of the representation theorems of Sections 3 and 4.

In the last section we then discuss an application of these results to the
problem of determining Markov quasi-factors of measure-preserving systems with
quasi-discrete spectrum. We show that under certain algebraic
assumptions on the signature of a system each Markov quasi-factor of
the system is necessarily a
factor.

\bigskip

\noindent
{\bf Terminology and Notation.}\
Throughout this article we generically write $\bfK = (K;\vphi)$ for 
topological and $\bfX = \MDSX$ for  measure-preserving dynamical systems.  
This means that in the first case $K$ is a compact Hausdorff space
and $\vphi: K \to K$ is continuous, while in the second case
$\prX = (X, \Sigma_\prX, \mu_\prX)$ is a probability space
and $\vphi: X \to X$ is a measure-preserving measurable map. 
The topological system $\bfK$ is called {\emdf separable} if  
$\Ce(K)$ is separable
as a Banach space, which is equivalent to $K$ being metrizable. 
Analogously, the measure-preserving system $\bfX$ is {\emdf separable} if
$\Ell{1}(\prX)$ is separable as a Banach space. This is equivalent to 
$\Sigma_\prX$ being countably generated (modulo null sets).

The corresponding {\emdf Koopman operators} on $\Ce(K)$
in the topological and on $\Ell{1}(\prX)$ in the 
measure-preserving situation are generically denoted by $T_\vphi$ or,
if the dynamics is understood, simply by $T$.

In general, our terminology and 
notation is the same as in \cite{EFHN}. 
In particular, if $T$ is a bounded operator on a complex Banach space $E$, we
write
\[ \Pspec(T) :=  \{ \lambda \in \C \st \text{$\lambda$ is an
  eigenvalue of $T$}\}
\]
for the {\emdf point spectrum} of $T$. 
Given two  measure-preserving systems $\bfX = (\prX; \vphi)$ and 
$\bfY = (\prY; \psi)$ we call each operator 
$S: \Ell{1}(\prX) \to \Ell{1}(\prY)$ a {\emdf Markov operator}
if it is one-preserving, order-preserving and integral-preserving.
A Markov operator $S$ is a {\emdf Markov embedding} 
if it is a lattice homomorphism,  a {\emdf Markov isomorphism} if it is a
surjective Markov embedding, and
{\emdf intertwining} if $ST_\vphi = T_\psi S$. 

Two systems $\bfX$ and 
$\bfY$ are {\emdf isomorphic} if there exists an intertwining
Markov isomorphism between the respective $\Ell{1}$-spaces. 
(For the connection with the notion of {\em point isomorphism}, see
\cite[Chapter 12]{EFHN}.)

A {\emdf factor} of a measure-preserving system 
$\bfX = (\prX; \vphi)$ is a measure-preserving
system $\bfY = (\prY; \psi)$ together with an intertwining 
Markov embedding $S: \Ell{1}(\prY) \to \Ell{1}(\prX)$.
Two factors 
\[ S_1 : \Ell{1}(\prY_1) \to \Ell{1}(\prX) \quad \text{and}\quad
S_2 : \Ell{1}(\prY_2) \to \Ell{1}(\prX)
\]
are considered {\em the same} if $\ran(S_1) = \ran(S_2)$ or, equivalently, 
if there is an intertwining Markov isomorphism 
$S: \Ell{1}(\prY_1)\to \Ell{1}(\prY_2)$ such that $S_2S = S_1$.
(See \cite[Section 13.4]{EFHN} for alternative descriptions of 
a factor.)

A {\emdf point factor map} of a system $\bfX$ to a system $\bfY$ 
is a measurable and measure-preserving map $\pi: X \to Y$ 
such that $\pi \after \vphi = \psi \after \pi$ almost everywhere. 
The associated Koopman operator 
\[ S : \Ell{1}(\prY) \to \Ell{1}(\prX),\qquad  Sf := f\after \pi
\]
is then an intertwining Markov embedding, and hence constitutes
a factor in our sense. By von Neumann's theorem, if
the underlying probability spaces are standard, then every
intertwining Markov embedding is induced by a point factor map,
cf.~\cite[Chapter 12]{EFHN} and, in particular, 
\cite[Appendix F]{EFHN}.

\section{Algebraic and Dynamic Preliminaries}\label{s.defs}

Let us start with some purely algebraic preparations.
The relevance of these will become  clear afterwards when we
turn to dynamical systems.

\subsection{Signatures}\label{ss.signatures}

\noindent
Suppose that $G$ is a (multiplicative) Abelian group and $\Lambda: G \to G$
is a homomorphism. Consider the homomorphism
\[ \Phi: G \to G, \quad \Phi(g) := g \cdot \Lambda g.
\]
Then the binomial theorem yields
\begin{equation}\label{defs.e.binom}
 \Phi^n g = \prod_{j=0}^n (\Lambda^j g)^{n\choose j} \qquad (g\in G).
\end{equation}
This is easy to see if one writes the group additively and notes
that in this case $\Phi = (\id + \Lambda)$.

Let us define the increasing chain of subgroups
\[ G_n := \ker(\Lambda^{n}) \qquad (n \ge 0),
\]
so that $G_{0} = \car$.
Then $\Lambda: G_n \to G_{n-1}$ for $n\ge 1$. Moreover, $\Phi$ restricts
to an automorphism on each $G_n$. (This is again easily seen 
by writing the group additively; the ``Neumann series''
\[ \Phi^{-1} = (\id + \Lambda)^{-1} = \sum_{j=0}^\infty (-1)^j \Lambda^j.
\]  
terminates when applied to elements of $G_n$, and yields the inverse 
$\Phi^{-1}$ of $\Phi$.)   

Recall that $\Lambda$ is called {\emdf nilpotent} if $G = G_n$ for some $n \in \N$.
We call $\Lambda$ {\emdf quasi-nilpotent} if $G = \bigcup_{n\ge 0} G_n$. 
Note that if $\Lambda \in \End(G)$ is (quasi-)nilpotent, then so is its
restriction to $\Lambda(G)$, as well as  the induced homomorphism 
(by abuse of language)
\[ \Lambda: \fact{G}{G_n} \to \fact{G}{G_n}, \quad \Lambda(gG_n) :=
(\Lambda g) \, G_n
\]
for every $n\in \N$. 
%A pair $(G, \Lambda)$ where 
%$G$ is a group and $\Lambda \in \End{G}$ is quasi-nilpotent
%is called a {\emdf quasi-nilpotent system}. 

%\begin{rem}

%\end{rem}

\medskip

Recall that a group $H$ is called {\emdf torsion-free} if
$H$ has no elements of finite order other than the neutral element.
%For signatures, we have the following  easy-to-prove characterization.

\begin{lemma}\label{defs.l.torsion-free}
Let $H$ be an Abelian group and $\Lambda: H\to H$  a quasi-nilpotent
homomorphism, with associated subgroups
$H_n = \ker(\Lambda^{n})$  as above. Then the induced homomorphism
\[ \Lambda: \fact{H_{n{+}1}}{H_n} \to \fact{H_{n}}{H_{n{-}1}}
\]
is injective.  Moreover, the following assertions are equivalent:
\begin{aufzii}
\item $H_1$ is torsion-free.
\item $H$ is torsion-free.
\item $\fact{H_{n+1}}{H_n}$ is torsion-free for every $n \ge 0$. 
\end{aufzii}
Moreover, if {\upshape (i)--(iii)} hold, then, with $\Phi$ defined as above,
for each $m \ge 1$ and $h\in H_{m+1} \without H_m$ the elements
\[ \Phi^n h, \quad n \ge 0
\]
are pairwise distinct modulo $H_{m-1}$. 
\end{lemma}

\begin{proof}
The first assertion follows directly from the
definition of $H_n$ and 
renders straightforward the proof of the stated equivalence. For the remaining 
assertion, note that it follows from the
binomial formula \eqref{defs.e.binom} that for all $m\ge 1$ and $n \ge 0$ and 
$h\in H_{m{+}1}$  one has
\[ \Phi^n h = h (\Lambda h)^n  \quad \mod{H_{m{-}1}}.
\]
Hence, if $n \ge k \ge 0$ and $\Phi^n h = \Phi^k h$ $\mod{H_{m{-}1}}$, 
then 
$(\Lambda h)^{n-k} = 1$ $\mod{H_{m{-}1}}$, which 
implies $n = k$, by (iii) and the fact that $\Lambda h \notin H_{m{-}1}$.  
\end{proof} 

A triple $(G, \Lambda, \iota)$ is called a {\emdf \signature} if 
$G$ is an Abelian group, $\Lambda : G \to G$ is a quasi-nilpotent homomorphism 
and 
\[ \iota: G_1 \to \torus
\]
is a monomorphism (= injective homomorphism), 
where $G_1 = \ker(\Lambda)$ as above.  
The \textbf{order} of the signature $(G, \Lambda, \iota)$ is
\[
\ord(G,\Lambda,\iota):=\inf\{n\in \N \st G = G_n\} \in \N \cup\{\infty\},
\]
in the sense that the order is infinite if $G \neq G_n$ for all $n \in \N$, i.e.,
if $\Lambda$ is not already nilpotent.

From the original \signature\ $(G, \Lambda, \iota)$ one can canonically 
derive new signatures. First, one can pass to $(\Lambda(G), \Lambda, \iota)$
where we write, for simplicity, again $\Lambda$ and $\iota$ for
the respective restrictions of $\Lambda$ to $\Lambda(G)$ and
$\iota$ to $\Lambda(G_2) \leq G_1$.

Second, for
each $n \in \N_0$ we obtain a derived signature 
$(\fact{G}{G_n}, \Lambda, \widetilde{\iota})$
in the following way.
The homomorphism $\Lambda$ canonically induces a monomorphism (again 
denoted by $\Lambda$) at each step in the following chain:
\[ %\ker(\Lambda_n) = 
\fact{G_{n+1}}{G_n} \stackrel{\Lambda}{\pfeil} 
\fact{G_n}{G_{n-1}} \stackrel{\Lambda}{\pfeil} 
\dots \stackrel{\Lambda}{\pfeil}  \fact{G_1}{G_0} = \fact{G_1}{\{\car\}} = G_1. 
\]
Hence
$\widetilde{\iota}:= \iota \circ  \Lambda^n:  \fact{G_{n+1}}{G_n} \to \torus$ 
is a monomorphism. But $\fact{G_{n+1}}{G_n}$ is precisely the kernel 
of $\Lambda$ when considered as a quasi-nilpotent 
homomorphism on $\fact{G}{G_n}$.
%In th, $(\fact{G}{G_n}, \Lambda_n, \widetilde{\eta_0})$ is indeed a \signature. 

\medskip

A {\emdf morphism} $\alpha: (G, \Lambda, \iota) \to (\tilde{G}, 
\tilde{\Lambda}, \tilde{\iota})$ of signatures
is every group homomorphism $\alpha : G \to \tilde{G}$ such that
$\tilde{\Lambda} \after \alpha = \alpha \after \Lambda$ and 
$\tilde{\iota} \after \alpha = \iota$ on $G_1$. It is then 
easily proved by induction that  
$\alpha\restrict{G_n}$ is injective for each $n \in \N$. Consequently,
{\em every morphism
$\alpha$ of signatures is injective.}
If $\alpha$ is bijective 
then its inverse is also a morphism of signatures, and 
$\alpha$ is an {\emdf isomorphism}.
For example, the derived signatures $(\Lambda(G), \Lambda, \iota)$ 
and $(\fact{G}{G_1}, \Lambda, \widetilde{\iota})$ are isomorphic
via the induced isomorphism $\Lambda: \fact{G}{G_1} \to \Lambda(G)$.

\subsection{Topological systems with quasi-discrete spectrum}\label{ss.tds-qds}

%\bigskip
Signatures arise naturally in the context of dynamical systems. 
Let $\bfK =(K; \psi)$ be a topological dynamical system with 
Koopman operator $T$ on $\Ce(K)$. Then the set 
\[ \Ce(K; \torus) = \{ f\in \Ce(K) \st \abs{f} = \car\}
\]
is an Abelian group, and 
\[ \Lambda_\bfK : \Ce(K; \torus)\to \Ce(K; \torus),\qquad 
\Lambda_\bfK f := Tf \cdot \conj{f}
\]
is a homomorphism, called the {\emdf derived homomorphism}. 
If $\bfK$ is understood, we simply  write $\Lambda$ in place of 
$\Lambda_\bfK$.  Note that  $Tf =f \cdot \Lambda f$, hence 
$T$ takes the role of $\Phi$ from above. In particular,  one has the formula
\begin{equation}\label{defs.e.T-binom}
 T^n f = \prod_{j=0}^n (\Lambda^j f)^{n \choose j} 
\end{equation}
for each $f\in \Ce(K;\torus)$ and $n \ge 0$.

Define $G_n(\bfK) := \ker(\Lambda^{n})$ for $n\ge 0$ and
\[ G(\bfK) := \bigcup_{n\ge 0} G_n(\bfK) = \bigcup_{n\ge 0} \ker(\Lambda^n). 
\]
(For simplicity, we write $G_n$ and $G$ if $\bfK$ is understood.)
Then $G(\bfK)$ is an Abelian group and $\Lambda= \Lambda_\bfK$ is
a quasi-nilpotent homomorphism on it. 
Note that $G_1 = \ker(\Lambda) = \fix(T) \cap \Ce(K;\torus)$.

The elements of $G_n$ are called (unimodular) {\emdf
  quasi-eigenvectors} of order $n{-}1$ (cf.{} Remark
\ref{defs.r.labelling} below), and $G = \bigcup_{n=0}^\infty G_n$ is
the group of all (unimodular) quasi-eigenvectors. Correspondingly,
each element of 
\[ H_{n{-}1} = H_{n{-}1}(\bfK) := \Lambda(G_n)
\]
is called a  {\emdf quasi-eigenvalue} of order $n{-}1$, and 
\[ H := H(\bfK) := \bigcup_{n\ge 0} H_{n} = \bigcup_{n \ge 1} \Lambda(G_n)
= \Lambda(G) 
\]
is the group of all quasi-eigenvalues. 
This terminology
derives from the fact that the elements of  $G_n$  are precisely the
unimodular solutions $f$ of an equation $Tf = gf$, where $g \in
G_{n{-}1}$
(in which case $g \in H_{n{-}1}$).

\medskip

Let us now suppose that {\em $\fix(T)$ is one-dimensional}, i.e., 
consists of the constant functions only. (This is the case, e.g.,
if $\bfK$ is a minimal system.) 
Then the group $G_2$ consists of 
all the unimodular eigenfunctions of $T$ corresponding
to unimodular eigenvalues, and $H_1 = \Lambda(G_2)$ is
the group of unimodular eigenvalues of $T$. (Indeed, since
$f\in G_2$, the function $\Lambda f$ is constant
and since $Tf = (\Lambda f) f$, $\Lambda f$ is a unimodular eigenvalue
with eigenfunction $f$. Conversely, if $Tf = \lambda f$ for some
nonzero function $f\in \Ce(K)$ and $\lambda \in \torus$, then 
$\abs{f} \in \fix(T)$, hence we can rescale and suppose without
loss of generality that $\abs{f} =1$, i.e. $f$ is unimodular. 
Then $\Lambda f = (Tf)  \conj{f} = \lambda \car \in G_1$ and
hence $f\in G_2$.)

Still under the hypothesis $\fix(T) = \C\car$, the  mapping
$\iota:= \iota_\bfK: \fix(T) \to \C$, which  maps a constant function to its value,
is an isomorphism of vector spaces. (Note that $\iota(f) = f(x_0)$ for 
$f\in \fix(T)$ and arbitrary $x_0\in K$.) 
Hence, its restriction $\iota: G_1  \to \torus$ to $G_1$ is an isomorphism of groups, and 
$(G, \Lambda, \iota)$ is a \signature.

The derived signature 
$\sig(\bfK) := (H, \Lambda, \iota)$
is called the {\emdf signature (of quasi-eigenvalues) of the system
$\bfK$}. Recall from above that this signature is, via $\Lambda$, 
isomorphic  to the signature  $(\fact{G}{G_1}, \Lambda, \widetilde{\iota})$. 
The topological system $\bfK$ is said to have
{\emdf quasi-discrete spectrum} if 
the linear hull of all quasi-eigenvectors is dense
in $\Ce(K)$, i.e., if
\[ \cspan^{\Ce(K)}(G) = \Ce(K).
\]

\subsection{Measure-preserving  systems with quasi-discrete spectrum}%
\label{ss.mps-qds}

A similar construction and terminology applies for a
measure-preser\-ving system $\bfX = \MDSX$ with Koopman operator $T$ on 
$\Ell{1}(\prX)$. Again one considers the {\emdf derived} group homomorphism
$\Lambda = \Lambda_\bfX$ defined by 
\[ 
\Lambda f := \Lambda_\bfX f := Tf \cdot \conj{f}
\]
on the Abelian group
\[ \Ell{0}(\prX; \torus) := \{ f \in \Ell{\infty}(\prX) \st
\abs{f} =1 \,\,\text{(a.e.)}\},
\]
which has kernel $G_1 := G_1(\bfX) := \fix(T) \cap \Ell{0}(\prX;
\torus)$,
the group of unimodular fixed functions.
As in the topological case,
we let  $G_n(\bfX) := \ker(\Lambda^{n})$ for $n\ge 0$ be the 
group of (unimodular) {\emdf quasi-eigenvectors} of order $n{-}1$ and
\[ G(\bfX) := \bigcup_{n\ge 0} G_n(\bfX) = \bigcup_{n\ge 0} \ker(\Lambda^n). 
\]
(Again,  we write $G_n$ and $G$ if $\bfX$ is understood.) Analogously,
\[ H_{n{-}1} = H_{n{-}1}(\bfX) = \Lambda(G_n)
\]
is the group of (unimodular) {\emdf quasi-eigenvalues} of
order $n{-}1$, and 
\[  H:= H(\bfX) := \Lambda(G(\bfX)) = \bigcup_{n=0}^\infty H_{n{-}1}
\]
is the group of all quasi-eigenvalues.

Now suppose that the system $\bfX$ is ergodic. Then all fixed
functions are essentially constant, so
$G_1 = \ker(\Lambda) = \{ c\car \st c\in \torus\}$.
Again, we denote by 
\[ \iota = \iota_\bfX: G_1 \to \torus, \qquad  \iota(c \car) = c
\]
the canonical monomorphism.  Then $(G, \Lambda, \iota)$ 
is a \signature. The derived \signature\
\[
\sig(\bfX) := (H, \Lambda, \iota)
\] 
is called the {\emdf \signature\ of the system $\bfX$}. 

The system $\bfX$ is said to have
{\emdf  quasi-discrete spectrum} if $G$ is a 
total subset of 
$\Ell{2}(\prX)$, i.e., if 
\[ \cspan^{\Ell{2}}(G) = \Ell{2}(\prX).
\]

The simplest nontrivial system with quasi-discrete spectrum 
is the skew shift. We describe this system and compute its signature below.

\begin{exa}[Skew shift]
\label{ex.mps-qds.ss}
Let $\prX$ be the two-dimensional torus $\T^2$ (written additively mod $1$) 
with the Lebesgue measure, and $\vphi$ be the transformation 
\[
\vphi: \torus^2 \to \torus^2 , \qquad \vphi(x,y) := (x+\alpha ,x+y )
\]
for some irrational $\alpha \in (0,1)$. 
The associated measure-preserving system  $\bfX = (\prX;\vphi)$ 
is called the skew shift. It is known that 
the skew shift is totally ergodic. (Cf.~\cite[Prop.~10.17]{EFHN}
for a proof of ergodicity; this proof can easily be adapted to
yield even total ergodicity.)

Write $\ue_k(t) := \ue^{2\upi \ui k t}$ for $t\in (0,1)$. Then some
computation shows that 
\[ G_2(\bfX) = \{ \lambda \ue_k\tensor \car \st k \in \Z,\: \lambda \in \torus\}
\]
and
\[ 
 G_3(\bfX) = \{ \lambda \ue_k \tensor \ue_l \st k,\:l \in \Z,\: \lambda\in 
\torus \}.
\]
It follows that $\bfX$ has quasi-discrete spectrum and $G(\bfX) = G_3(\bfX)$
(see Corollary \ref{fac.c.orth} below). 
Another little computation yields
\[ \Lambda(\lambda \ue_k \tensor \ue_l) = \ue_k(\alpha) (\ue_l \tensor \car)
 \qquad (k, l \in \Z,\: c\in \torus)
\]
from which it follows that 
\[ H_0(\bfX) = \{ \mathbf{1}\},\qquad 
H_1(\bfX) = \{ \ue_k(\alpha) (\car \tensor \car) \st  k \in \Z\}
\]
and
\[ H(\bfX) = H_2(\bfX) = \{ \ue_k(\alpha) (\ue_l \tensor \car) \st
 k,l \in \Z\}.
\]
This means that $\sig(\bfX) \simeq (\Z^2 , \Lambda, \iota)$, where
$\Lambda$ is the two-step nilpotent homomorphism
\[ \Lambda: \Z^2 \to \Z^2, \qquad \Lambda(k,l) = (l,0)
\]
and $\iota: \Z \to \torus$ is the embedding given by 
$\iota(k,0):= \ue_k(\alpha)$.
\end{exa}

\begin{rem}\label{defs.r.labelling}
Let us stress the fact that our terminology deviates slightly from the
standard one (established first by Abramov and continued by  successors).
In Abramov's work, the labelling of the groups $G_n$ is shifted to the
effect that what we call $G_{n}$ would be $G_{n-1}$ in Abramov's
terminology. We have chosen for this change in order to have a unified 
labelling for the significant subgroups associated with a quasi-nilpotent
endomorphism of an Abelian group.

Other authors (e.g. Lesigne \cite{Lesigne1993}) 
in the case of an ergodic
system $\bfX$ write
$E_0(T)$ for the set of eigenvalues and define recursively
\[ E_k(T) = \{ f\in \Ell{0}(\prX;\torus) \st \Lambda f  \in E_{k{-}1}(T)\}
\]
for $k \ge 1$. This means  that 
\[ E_0(T) = H_1(\bfX)\quad  \text{and} \quad E_k(T) = G_{k{+}1}(\bfX)
\quad \text{for $k\ge 1$}
\]
if the system $\bfX$ is ergodic.
\end{rem}

\subsection{Affine automorphisms}\label{ss.affauto}

\noindent
Let $\Gamma$ be a compact Abelian group, $\eta\in \Gamma$ and 
$\Psi: \Gamma  \to \Gamma$ a continuous automorphism of $\Gamma$.  Then the
mapping 
\[  \psi: \Gamma \to \Gamma, \qquad  \psi(\gamma) := \Psi(\gamma)\cdot \eta
\]
is called an {\emdf affine automorphism}. The topological dynamical system
$(\Gamma;\psi)$
 is called an {\emdf affine automorphism system}, and denoted by 
$(\Gamma; \Psi, \eta)$. 
Clearly, the Haar measure is invariant under $\psi$, hence this gives rise
also to a measure-preserving system $(\Gamma, \haar;\Psi,\eta)$.

Suppose that $H$ is a (discrete) Abelian group, $\Lambda \in \End(H)$ 
is quasi-nilpotent  and $\eta \in H^*$, the (compact) dual group.  
Then $\Phi : H \to H$, defined by $\Phi(h) = h \Lambda h$ is an automorphism
of $H$. Passing to the dual group $H^*$ we obtain the dual automorphism
$\Phi^*\in \Aut(H^*)$, and $(H^*; \Phi^*,\eta)$ is an 
affine automorphism system. 

The following result says that the conjugacy class of such an 
affine automorphism system is determined by $\Phi$ and the restriction 
of $\eta$ to $H_1$.

\begin{thm}\label{defs.t.sig-qds-conj}
Let $H$ be a (discrete) Abelian group, $\Lambda: H \to H$  a quasi-nilpotent homomorphism, 
with induced automorphism $\Phi = \Id \cdot \Lambda$ as above, and  let $\eta \in H^*$. 
Then, for fixed $\gamma\in H^*$ the  rotation map 
\[ R_{\gamma} : H^* \to H^*,\qquad  \chi \mapsto \chi \gamma
\]
induces an isomorphism (conjugacy) 
\[ R_\gamma : (H^*; \Phi^*, \eta \Lambda^*\gamma)  \stackrel{\cong}{\longrightarrow} 
(H^*; \Phi^*, \eta) 
\]
of (topological) affine automorphism systems. Moreover, the set
\[\{ \eta \,\Lambda^*\gamma \st \gamma\in H^*\}
\]
consists precisely of those $\chi \in H^*$ which coincide with $\eta$ on $H_1$.
\end{thm}

\begin{proof}
The first assertion is established by a straightforward computation.
For the second, note first that $\Lambda^*\gamma = \gamma \after \Lambda = 1$ 
on $H_1$. 
Hence $\eta \Lambda^*\gamma = \eta$ on $H_1$. Conversely suppose that 
$\chi \in H^*$ and $\chi = \eta$ on $H_1$. Then $\chi\eta^{-1} = 1$ on 
$H_1$, hence
one can define 
$\gamma_1: \Lambda(H) \to \torus$ by 
\[ \gamma_1 ( \Lambda(h)) := (\chi \eta^{-1})(h) \qquad (h \in H).
\]
Now, if $\gamma\in H^*$ extends $\gamma_1$, then 
$\chi = \eta \Lambda^*\gamma$ as claimed. Note that 
such an extension always exists since $\torus$ is divisible,
cf.{}  also \cite[Prop.{}14.27]{EFHN}
\end{proof}

If $(H, \Lambda, \iota)$ is a signature and $\eta : H \to \torus$ is
any homomorphism that extends $\iota$, then the affine automorphism
system $(H^*; \Phi^*, \eta)$ is called {\emdf associated} 
with the signature $(H, \Lambda, \iota)$. By the result above, all
affine automorphism systems associated with the same signature are
topologically conjugate. Since the topological conjugacy is a rotation
and hence preserves the Haar measure, it is also a conjugacy
for the measure-preserving systems.

By the results of Hahn--Parry and Abramov
(see Theorems \ref{tqds.t.real} and \ref{mqds.t.real} below), if
$(H, \Lambda,\iota)$ is a signature such that  $H_1$ is torsion-free, then 
any associated topological system $(H^*; \Phi^*,\eta)$ as well as
the corresponding measure-preserving system $(H^*, \haar; \Phi^*,\eta)$
has quasi-discrete spectrum with signature $(H, \Lambda,\iota)$.

\section{Topological systems with quasi-discrete spectrum}\label{s.tqds}

From now on, we let $\bfK = (K; \psi)$ be a topological system 
such that $\fix(T)$ is one-dimensional, where $T$ is, as always, the
corresponding Koopman operator on $\Ce(K)$.  
Suppose that $\bfK$ has quasi-discrete spectrum 
with the additional property that the
group of eigenvalues
\[ H_1 = \Lambda(G_2) \cong \fact{G_2}{G_1} \lesssim  \torus
\]
is torsion-free. 
Equivalently, by Lemma \ref{defs.l.torsion-free},
the group $H$ of all quasi-eigenvalues is torsion-free.
Under these hypotheses we obtain the following result,
obtained first by Hahn and Parry \cite{HahnParry1965}.

\begin{lemma}\label{tqds.l.asymp}
Let $\bfK$ be a topological system with quasi-discrete spectrum
such that $\dim\fix(T) =1$ and the group $H_1$ of unimodular eigenvalues is
torsion free. Then
\[ 
\lim\limits_{N\to \infty} \frac 1 N \sum\limits_{j=0}^{N-1} (T^j f)(x) = 
0 \quad
\text{for every $x\in K$ and $f\in G\setminus G_1$.}  
%\begin{cases} 
%0 &  \text{for  } f\in \Gamma\setminus \torus,\\
%f &  \text{for } f\in \torus.
%\end{cases}   
\]
\end{lemma}

\begin{proof}
Let $f \in G_{k{+}1} \setminus G_k$ for some $k \ge 1$. For $n \ge k$ and 
$x\in K$, 
\[ T^n f(x) = \prod_{j=0}^n (\Lambda^j f)^{n \choose j }(x)
=   \prod_{j=0}^{k} [(\Lambda^j f)(x)]^{n \choose j} = f(x) \ue^{2\upi \ui p_x(n)},
\]
where $p_x(n) = \sum_{j=1}^k { n\choose j} \theta_j(x)$ and 
$\theta_j(x) \in \R$
is such that $(\Lambda^j f)(x) = \ue^{2\upi \ui \theta_j(x)}$. 
The leading 
coefficient of the polynomial $p_x$ is $\theta_k(x)/k!$, and this is
irrational since $\Lambda^k f \in H_1$ and $H_1$ is torsion-free. 
By Weyl's equidistribution theorem, 
\[ \lim_{N\to \infty} \frac{1}{N} \sum_{n=0}^{N-1} (T^n f)(x) =
 f(x) \lim_{N\to \infty} \frac{1}{N} \sum_{n=0}^{N-1} \ue^{2\upi \ui p_x(n)} = 0.
\]
\end{proof}

The lemma yields immediately the following theorem.

\begin{thm}\label{tqds.t.unierg}
Let $\bfK$ be a topological system with quasi-discrete 
spectrum such that $\dim\fix(T)=1$ and
the group of quasi-eigenvalues is torsion-free. Then $\bfK$ is 
uniquely ergodic. Moreover, elements of $G$ that are
different modulo $G_1$ are orthogonal with respect to 
the unique invariant probability measure.
%, and let $(K,\mu;\psi)$ be
%as above. Then the topological system $(K,\psi)$ is strictly ergodic, 
%i.e., it is minimal and  uniquely ergodic and its unique invariant
%probability measure has full support. 
\end{thm}

Note that a uniquely ergodic system has a unique minimal subsystem
(as every minimal subsystem is the support of an invariant measure, 
see \cite[Chapter 10]{EFHN}). Hence, we shall suppose in the following that 
$\bfK$ is {\em minimal}.

\begin{lemma}\label{qds.l.totmin}
Let $\bfK$ be a minimal topological system with quasi-discrete spectrum. 
Then its  group $H_1$ of unimodular eigenvalues is torsion-free if and only 
if $\bfK$ is totally minimal. 
\end{lemma}

\begin{proof}
If $\bfK$ is any totally minimal topological system, 
then every power $T^m$ of its Koopman
operator has one-dimensional fixed space. Hence, the group of unimodular
eigenvalues is torsion-free.

Conversely, let $\bfK= (K; \psi)$ be a minimal 
system with quasi-discrete spectrum, such that 
$H_1$ is torsion-free. By Theorem \ref{tqds.t.unierg}, $\bfK$ has a unique
invariant probability measure $\mu$, say, which has full support. 
Now let, as above, be $T$ the Koopman operator on $\Ce(K)$ of $\bfK$, and 
fix $m \in \N$. Any non-constant function in $\fix(T^m)$ would lead
to $T$ having an unimodular eigenvalue of order $m$, which is excluded. 
Hence, also $\fix(T^m)$ is one-dimensional. Moreover, it is easy
to see from formula \eqref{defs.e.T-binom} that any quasi-eigenfunction 
for $T$ is also a quasi-eigenfunction for $T^m$. It follows that 
also the system $(K; \psi^m)$ has quasi-discrete spectrum. 
The corresponding group of quasi-eigenvalues is a subgroup of $H_1$, hence
torsion-free. By Theorem \ref{tqds.t.unierg}, $(K, \psi^m)$ is uniquely 
ergodic, and  since $\mu$ is $\psi^m$-invariant and has full support, 
it follows that $(K; \psi^m)$ is minimal.
\end{proof}

In the next step we show that a {\em totally minimal} 
topological system $\bfK$ of quasi-discrete spectrum  
is isomorphic to a specific affine automorphism system on a 
compact monothetic group.

\begin{thm}[Representation]\label{tqds.t.rep}
Let $\bfK= (K;\psi)$ be a totally minimal topological system with 
quasi-discrete spectrum and \signature\  $(H, \Lambda, \iota)$. Then 
$\bfK$ is isomorphic to the affine automorphism system $(H^*; \Phi^*, \eta)$, 
where $\Phi(h) = h \Lambda h$ for $h \in H$, and $\eta$ denotes
any homomorphic  extension of $\iota: H_1 \to \torus$ to all of $H$. 
\end{thm}

\begin{proof}
By Theorem \ref{defs.t.sig-qds-conj} 
it suffices to find {\em one} extension
$\eta \in H^*$ of $\iota$ such 
that $\bfK$ is isomorphic to $(H^*; \Phi^*, \eta)$. 
The proof will be  given in several steps and employs the
isomorphy (via $\Lambda$) of $\fact{G}{G_1}$ and $H$. 

Fix $x_0 \in K$ and consider for each $x\in K$
the multiplicative functional 
\[ \delta_x: \Ce(K) \to \C,\qquad \delta_x(f) = \frac{f(x)}{f(x_0)}.
\]
This restricts to a homomorphism $\delta_x: G \to \torus$ that factors through $G_1$, 
hence induces a homomorphism $\delta_x : \fact{G}{G_1} \to \torus$, i.e.,
$\delta_x\in (\fact{G}{G_1})^*$. We claim that the mapping
\[ \delta: K \to (\fact{G}{G_1})^*, \qquad x \mapsto \delta_x
\]
is a homeomorphism. 

\prfskip
Since $\Gamma := (\fact{G}{G_1})^*$ carries the topology of pointwise
convergence on $\fact{G}{G_1}$, $\delta$ is continuous. Since $G$
separates the points of $K$, $\delta$ is injective. For the surjectivity it suffices
to show that the
induced Koopman operator
\[ \Delta: \Ce(\Gamma) \to \Ce(K), \qquad (\Delta f)(x) := f( \delta_x).
\]
is injective. To this end, note that 
$\{ gG_1 \st g \in G\} = \fact{G}{G_1} \cong \Gamma^*$ and, with this identification 
$\Delta(gG_1) = \frac{g}{g(x_0)}$ for $g\in G$. Moreover, 
by Theorem \ref {tqds.t.unierg}, if $gG_1 \neq hG_1$ then 
$g \perp h$ in $\Ell{2}(K;\mu)$. Hence, $\Delta : \spann(\Gamma^*) \to \Ce(K)$ 
is an isometry for the  $\Ell{2}$-norms, i.e.,
\[ \norm{f}_{\Ell{2}(\Gamma;\haar)} = 
\norm{\Delta f}_{\Ell{2}(K,\mu)} \le \norm{\Delta f}_{ \Ce(K)} 
\qquad( f\in \spann(\Gamma^*)).
\]
Since the $\Ell{2}$-norm on $\Ce(\Gamma)$ is weaker than the uniform norm
and $\spann(\Gamma^*)$ is dense in $\Ce(\Gamma)$, it follows by approximation that
\[ \norm{f}_{\Ell{2}(\Gamma;\haar)} \le \norm{\Delta f}_{ \Ce(K)}
\]
for {\em all} $f\in \Ce(\Gamma)$. 
And, since the $\Ell{2}$-norm is really a norm on $\Ce(\Gamma)$, 
i.e., the Haar measure has full support,  $\Delta$ is injective. 

\prfskip
Finally,  we can --- via the mapping $\delta$ ---  
carry over the action $\psi$ from $K$ to $\Gamma$. 
For $x\in K$,
\begin{align*}
\delta_{\psi(x)}(f) &  = \frac{f(\psi(x))}{f(x_0)}
= \frac{(Tf)(x)}{f(x_0)} = \frac{f(x)}{f(x_0)}  (\Lambda f)(x)
= \frac{f(x)}{f(x_0)}  \frac{(\Lambda f)(x)}{(\Lambda f)(x_0)} 
(\Lambda f)(x_0)
\\ & = \delta_x(f) \cdot \delta_x(\Lambda f) \cdot (\Lambda f)(x_0)
= (\delta_x \Lambda^*\delta_x)(f) \cdot (\Lambda f)(x_0).
\end{align*}
This means that
\[ \delta_{\psi(x)}  = \Phi^*(\delta_x) \eta
\]
where $\eta(fG_1) = (\Lambda f)(x_0)$ for $f\in G$. Note that 
$\eta$ restricts on $\fact{G_2}{G_1}$ to the canonical embedding of 
$\fact{G_2}{G_1} \to \torus$. Hence, the theorem is completely proved.  
\end{proof}

\begin{cor}[Isomorphism]
Two minimal topological systems with quasi-discrete spectrum
and torsion-free group of unimodular eigenvalues are
conjugate if and only if  their signatures are isomorphic.
\end{cor}

\begin{proof}
It is clear that two conjugate systems have isomorphic signatures. 
Conversely, any isomorphism
of the associated signatures induces an isomorphism of associated
affine automorphism systems, and by Theorem \ref{tqds.t.rep} this leads
to an isomorphism of the original systems. 
\end{proof}

In order to complete the picture, only the following result is missing.

\begin{thm}[Realization]\label{tqds.t.real}
Let $(H, \Lambda, \iota)$ be a signature and consider 
an  
 associated affine automorphism system $\bfK := (H^*; \Phi^*, \eta)$. 
If $H$ is torsion-free
then $\bfK$ is totally minimal
and has quasi-discrete spectrum with signature (isomorphic to) 
$(H, \Lambda, \iota)$. 
\end{thm}

\begin{proof}
Denote by  $K := H^*$ and $\vphi: K \to K$, $\vphi  = \Phi^* \cdot \eta$. 
We denote, as usual, by $T$ the Koopman operator on $\Ce(K)$, and define
$\Lambda_T f := \conj{f} Tf$ for $f\in \Ce(K;\torus)$. 
The associated subgroups of $\Ce(K;\torus)$ are
\[ G_n = \ker(\Lambda_T^{n})\quad (n \in \N_0) \quad
\text{and}\quad G = \bigcup_{n\in \N} G_n.
\]
We consider $H$ as a subset of $\Ce(K;\torus)$. Define, 
\[ \Gamma_n := \torus \cdot H_{n{-}1} \quad(n \in \N) \quad
\text{and}\quad
\quad \Gamma = \bigcup_{n\in \N} \Gamma_n.
\]
A straightforward computation yields
\[  Th = h \circ \vphi = \eta(h) \cdot h \cdot \Lambda h = \eta(h) \Phi h,
\]
whence it follows that $\Lambda_Th = \eta(h) \Lambda h$ for $h\in H$. Consequently,
$\Gamma_n \subseteq G_n$ for all $n \in \N$. 

Let us compute the eigenspaces of $T$. Clearly, each $h\in H_1$ is
an unimodular eigenvector of $T$ with eigenvalue $\eta(h)$. Conversely,
suppose that there is $f\in \Ce(K)$ with $Tf = \lambda f$. 
Since $H$ is an  orthonormal basis of $\Ell{2}(K,\haar)$, the function
$f$ can be uniquely
written as an $\Ell{2}(K,\haar)$-convergent sum
\[ f = \sum_{h\in H} \lambda_h h.
\]
Applying $T$ leads to 
\[ \lambda \sum_{h\in H} \lambda_h h = \lambda f = Tf = \sum_{h\in H} \lambda_h \eta(h) \Phi h
\]
which, by comparison of coefficients, is equivalent to 
$\lambda \lambda_{\Phi h} = \lambda_h \eta(h)$ for all $h \in H$. 
In particular, 
\[ \abs{\lambda_{\Phi h}} = \abs{\lambda_h} \quad \text{for all $h\in H$}.
\]
If $h \notin H_1$, by Lemma \ref{defs.l.torsion-free} 
the set $\{ \Phi^n h \st n \ge 0\}$
is infinite. Hence, $\abs{\lambda_h} = 0$ if $h \notin H_1$. Moreover
$\lambda \lambda_h = \eta(h) \lambda_h$ for $h\in H_1$, 
which is equivalent with  $\lambda_h = 0$ or $\eta(h) = \lambda$. 
Since $\eta$ is a monomorphism on $H_1$, we conclude that $\sigma_p(T) = 
\eta(H_1)$ and  each eigenspace is one-dimensional and spanned
by a function of $H_1$. 

It follows  that $\bfK$ has quasi-discrete spectrum and its group
of eigenvalues is torsion-free. By Theorem \ref{tqds.t.unierg}, 
$\bfK$ is uniquely
ergodic, and since the Haar measure is invariant and has full support, 
$\bfK$ is minimal. By Lemma \ref{qds.l.totmin}, it is totally minimal. 

Recall that $G_1 = \{ c\car\st c\in \torus\}$ and 
consider the homomorphism of groups
\[  \alpha: H \to \fact{G}{G_1}, \qquad \alpha(h) := hG_1.
\]
Then $\alpha$ is 
injective, and it is easy to see that $\alpha: (H, \Lambda,\iota)
\to (\fact{G}{G_1}, \Lambda_T, \tilde{ \iota}))$ is a
monomorphism of signatures.  

It remains to be shown that $\alpha$ is an isomorphism, i.e., surjective. 
To this end, suppose that $g\in G$ is such that $gG_1 \notin \alpha(H)$. 
Then, by Theorem \ref{tqds.t.unierg} again, 
$g \perp H$ in $\Ell{2}(K, \haar)$. But $H$ is an orthonormal
basis, and hence $g = 0$, which is a contradiction to $\abs{g} = \car$.
\end{proof}

\begin{rem}
Theorem \ref{tqds.t.real} is due to 
Hahn and Parry \cite[Theorem 4]{HahnParry1965}. Our presentation
is more detailed, for the sake of convenience. 
\end{rem}

\section{Measure-preserving systems with quasi-discrete spectrum (Abramov's theorem)}\label{s.mqds}

We now turn to the measure-preserving case. Again, we start with the 
representation theorem. 
 
Let $\bfX = (\prX; \vphi)$ be a totally ergodic measure-preserving system
with quasi-discrete spectrum. Its 
Koopman operator on $\Ell{1}(\prX)$ is denoted  by  $T$, the
homomorphism $\Lambda$ on the group $\Ell{0}(\prX;\torus)$ is given by 
$\Lambda  f := \conj{f} \cdot Tf$, and as before the subgroup $G$
is given by 
\[ G = \bigcup_{n\ge 0} G_n, \qquad G_n := \ker(\Lambda^{n}) \qquad (n\ge 0).
\]
Since the system is totally ergodic, 
$G_1 = \ker(\Lambda) = \fix(T) \cap \Ell{0}(\prX;\torus)  
= \torus\cdot \car$, and the group of eigenvalues $\Lambda(G_1) \cong 
\fact{G_2}{G_1}$ is torsion-free. That $\bfX$ has 
{\emdf quasi-discrete spectrum} 
means 
that the linear hull of $G$, $\spann G$, is dense in  $\Ell{2}(\prX)$. 

Consider now the closure
\[ A := \cl_{\Ell{\infty}} \spann G
\]
of $\spann G$ in $\Ell{\infty}$. Since $G$ is multiplicative and 
$T$-invariant, $A$ is a $T$-invariant, unital  $C^*$-subalgebra of 
$\Ell{\infty}(\prX)$. Hence, by an application of Gelfand's 
theorem, we can find a
topological system $(K, \mu; \psi)$ together with a
Markov isomorphism $\Psi: \Ell{1}(K,\mu) \to \Ell{1}(\prX)$ such that 
$T\Psi = \Psi T_\psi$ and $\Psi(\Ce(K)) = A$. 
(See \cite[Chapter 12]{EFHN} for details.)
In the following
we identify $\prX$ with $(K,\mu)$ and  $A$ with $\Ce(K)$, 
drop explicit reference to the mapping $\Psi$, and 
write  again $T$  for the Koopman operator on $\Ce(K)$ of the mapping
$\psi$.  With these identifications being made, we now
have $G \subseteq \Ce(K; \torus)$, and hence $\bfK := (K;\psi)$ is
a topological system with quasi-discrete spectrum. The signature
$(H, \Lambda, \iota)$ 
of this topological system is, by construction, the same as
the signature of the original measure-preserving system. Moreover,
since the measure $\mu$ on $K$ has full support (also by construction),
the system $\bfK$ is minimal (Theorem \ref{tqds.t.unierg}), hence even totally
minimal by Lemma \ref{qds.l.totmin}.

Now we can apply Theorem \ref{tqds.t.rep} to conclude that 
$\bfK$ is isomorphic to the affine automorphism system 
$(H^*; \Phi^*, \eta)$, where $\eta$ is any homomorphic extension 
of $\iota$ to $H$. By virtue of this isomorphism,
the measure $\mu$ on $K$ is mapped to an invariant measure on $H^*$, 
which, by unique ergodicity of the systems, must therefore coincide
with the Haar measure on $H^*$. The isomorphism  of topological systems
therefore extends to an isomorphism $\bfX \cong (H^*, \haar; \Phi^*, \eta)$
of measure-preserving systems. 

In effect, we have proved the following theorem, due to Abramov 
\cite[\S 4]{Abramov1962}.

\begin{thm}[Representation]\label{mqds.t.rep}
Let $\bfX= (\prX;\psi)$ be a totally ergodic measure-preserving system with 
quasi-discrete spectrum and \signature\  $(H, \Lambda, \iota)$. Then 
$\bfX$ is isomorphic to the affine automorphism system 
$(H^*,\haar; \Phi^*, \eta)$, 
where $\Phi(h) = h \Lambda h$ for $h \in H$, and $\eta$ denotes
any homomorphic  extension of $\iota: H_1 \to \torus$ to all of $H$. 
\end{thm}

As in the topological case, the representation theorem implies
readily the isomorphism theorem. The proof is completely analogous.

\begin{cor}[Isomorphism]
Two totally ergodic measure-preserving  systems with quasi-discrete spectrum
are isomorphic if and only if  their signatures are isomorphic.
\end{cor}

\begin{rem}
Recall that  the notion of isomorphism used here is that
of a {\em Markov isomorphism}, see Introduction. By a famous theorem
of von Neumann, see \cite[Appendix E]{EFHN}, if the underlying
measure spaces are standard Lebesgue spaces, then Markov isomorphic  
systems are point isomorphic. Since a  system $\bfX$ is
Markov isomorphic to a standard Lebesgue system if and only if
it is separable, restricting the results to standard
Lebesgue spaces amounts to considering only signatures $(H, \Lambda, \iota)$
with a {\em countable} discrete group $H$.  
\end{rem}

Finally, as in the topological case, we complete the picture with
the realization result. Its proof is---mutatis mutandis---the same  as
the proof of Theorem \ref{tqds.t.real}.

\begin{thm}[Realization]\label{mqds.t.real}
Let $(H, \Lambda, \iota)$ be a signature such that $H$ is torsion-free.
Then any associated (as above) measure-preserving affine automorphism system 
$\bfX := (H^*, \haar; \Phi^*, \eta)$ is totally ergodic and
has quasi-discrete spectrum with signature (isomorphic to) 
$(H, \Lambda, \iota)$. 
\end{thm}

\subsection*{Final Considerations}\label{ss.finalcon}

With the representation theorems at hand, one can confine 
to systems of the form $\bfK = (H^*; \Phi^*, \eta)$ 
(and their measure-theoretic analoga) 
when studying the fine
structure of totally minimal/er\-godic systems with quasi-discrete spectrum.

As $H$ is the inductive limit of the $\Lambda$-invariant
subgroups $H_n$, 
the system $\bfK$ is the inverse limit of the systems $(H_n^*; \Phi^*, \eta)$. 
We shall briefly indicate that each step in this chain is an abstract
compact group extension by a continuous homomorphism. 

The canonical embedding $H_n \subseteq H_{n+1}$ induces a
canonical continuous epimorphism $H_{n{+}1}^* \to H_n^*$
with kernel
\[  F := \{ \gamma \in H_{n+1}^* \st \gamma\res{H_{n}} = \car \}. 
\]
Since all groups are Abelian, the compact subgroup $F$
of $H_{n+1}^*$ acts by multiplication as   automorphisms of the
affine rotation system. Indeed, for 
$\gamma \in F$ one has $\Lambda^*\gamma = \gamma \after \Lambda = 1$
on $H_{n+1}$ and  hence 
\[    \Phi^*(\chi \gamma) = (\chi \gamma) \Lambda^*(\chi \gamma) \eta
= \big(\chi (\Lambda^*\chi) (\Lambda^*\gamma) \eta \big)\,\,\gamma
= \big(\chi (\Lambda^*\chi) \eta \big)\,\,\gamma
= \Phi^*(\chi) \,\,\gamma
\]
for all $\chi\in H_{n+1}^*$. Hence $H_n^* \cong H_{n{+}1}^\ast/ F$ 
not just as compact groups, but also as affine rotation systems.

It follows (e.g. by \cite[Prop.{} 6.6]{Ellis1969}, but
the proof can be simplified because of minimality) 
that each totally minimal system with quasi-discrete spectrum 
is distal. %and hence its topological entropy is zero. 

\medskip

From this one can eventually prove that 
{\em every totally ergodic system with quasi-discrete
spectrum has zero entropy}. In fact, the proof is rather straightforward 
under the additional assumption
that the system is separable, i.e., its
$\Ell{1}$-space is separable. In that case,
the group $G/G_1$ (notation from above) has to be countable
by  Theorem \ref{tqds.t.unierg}.
Consequently, the $\Ell{\infty}$-closed linear span 
of $G$ is a separable
$C^*$-algebra, hence its Gelfand space is metrizable. To sum up, 
the original system has a totally minimal and metrizable model which, as
seen above, is distal. As Parry has shown in \cite{Parry1968} (see also
\cite[Chap.{} 4, Thm.{} 17]{Parry1981})
such systems have zero entropy. (Compare this proof with Abramov's from 
\cite[\S 5]{Abramov1962}.)

In the general case, i.e., if $\Ell{1}(\prX)$ is not separable, one may want to  
take advantage of the fact 
that the measure-theoretic entropy of a system is the supremum of
the entropies of its separable factors. However, we do not see how to proceed from here
directly without any further knowledge about the factors of
a system with quasi-discrete spectrum. 

It is our goal in the following section, built on
\cite{HahnParry1968}, to provide such knowledge.
We shall obtain a proof of the general statement
--- that 
{\em every} totally ergodic system with quasi-discrete
spectrum, separable or not,  has zero entropy --- which does not use any of the results of the present and the preceding section.

%One can show --- in fact by much more elementary methods independent of the results
%from above --- that 
%{\em every} totally ergodic system with quasi-discrete
%spectrum has zero entropy, no matter whether its $\Ell{1}$-space is
%separable or not. This is seen in the end of the following section. 

\section{Factors of systems with quasi-discrete spectrum}\label{s.fac}

In this  section, which is completely independent of Sections 3 and 4, 
 we study
factors of systems with quasi-discrete spectrum, recovering and extending
results from \cite{HahnParry1968}.

\medskip

\subsection*{A Technical Result}\mbox{}

\medskip

Let $\bfX$ and $\bfY$ be   measure-preserving
systems such that  $\bfY$  is a factor of $\bfX$. 
As is explained in \cite[Sec.{} 13.3]{EFHN}
one can consider the space $\Ell{2}(\prY)$ as being
a $T$-invariant subspace (in fact: a closed Banach sublattice containing the constants) of 
$\Ell{2}(\prX)$. (Note that we do not require the dynamics to be invertible, 
and even if it was, our notion of a factor only requires $T$-invariance and 
not $T$-bi-invariance. A $T$-bi-invariant factor is called a {\emdf strict factor},
see \cite[Sec.{} 13.4]{EFHN}.)

For simplicity, we shall abbreviate
\[ X := 
\Ell{2}(\prX) \quad \text{and}\quad Y : = \Ell{2}(\prY) \subseteq \Ell{2}(\prX).
\]
This is evidently an abuse of language, since usually $X$ and $Y$ denote the sets
of the underlying probability spaces. However, base-space maps do not occur in this section,
all arguments are purely operator theoretic, and it  is better to have simple symbols
for the function spaces rather than for the underlying sets. 

Following this philosophy, we denote by 
\[ \Cex{\cdot}{Y} : X \to Y
\]
the conditional expectation  (=Markov projection) onto the ($\Ell{2}$-space of the) factor.
It is an easy exercise to establish the identity
\[ T\Cex{f}{Y} = \Cex{Tf}{TY} \qquad (f\in X)
\]
where $TY = T(Y) = \{ Tf \st f\in Y\}$, a factor as well. 

Recall from above the abbreviation $\Lambda f := \conj{f}Tf$ and 
%As $\bfX$ has quasi-discrete spectrum, it is generated by its group 
%$G$ of quasi-eigenvectors. Recall that $G = \bigcup_{n\in \N} G_n$ where
\[ G = \bigcup_{n\ge 0} G_n ,\quad \text{where}\quad
G_n = \{ g\in X \st \abs{g}=1\,\,\text{and}\,\,\Lambda^n g = \car\}.
% \quad \text{and}\quad \Lambda f := \conj{f}Tf.
\]
The main technical result of Hahn and Parry from \cite{HahnParry1968} 
is the following. We shall provide a new proof.

\begin{thm}\label{fac.t.HPmain}
Let $\bfX$ be a measure-preserving system with Koopman operator $T$ and  
let $H$ be a subset of $G(\bfX)$ such that 
\begin{equation}\label{fac.eq.HPmain}
 h\in H\quad \Rightarrow\quad \conj{h}T^n h \in H \quad \text{ for
all $n\ge 1$.}
\end{equation}
Let  $\bfY$ be a factor of $\bfX$ with the following property:
whenever $h\in H$ and $f\in Y$ are such that $hf\in \bigcup_{n\in \N} \fix(T^n)$, then  $hf$ is a constant. Then, for each $h\in H$ either 
$h\in Y$ or $h \perp Y$. 
\end{thm}

\begin{rems}\label{fac.r.HPmain}
\begin{aufziii}
\item In Hahn and Parry's original formulation, $H$ was required to
be a $T$-invariant subgroup of $G(\bfX)$. 

\item It follows from  the representation
\[ \conj{h}T^nh = \Lambda( h \cdot Th \cdot T^2h \cdots T^{n{-}1}h)
\]
that $H$ satisfies the condition \eqref{fac.eq.HPmain}
if $H$ is $\Lambda$-invariant and 
for $h \in H$ one has  $h\: Th \cdots T^nh \in H $ for all $n \ge 1$. 

\item 
If $H \neq \emptyset$ then $\car \in H$, so 
a factor  $\bfY$ satisfying the hypotheses of the theorem  
is necessarily totally ergodic. And if $\bfY$ is totally
ergodic, a function $f$ as in the theorem 
has necessarily constant modulus. 

\end{aufziii}
\end{rems}

For the proof we introduce the notation 
$H_n = H\cap G_n$ for $n\in \N_0$, so that 
$H = \bigcup_{n\ge 0} H_n$. 
Since  $H$ is $\Lambda$-invariant,  
$\Lambda$ maps $H_{n{+}1}$ into $H_n$.
Note that $T$ is not assumed to be invertible on $Y$. Therefore we introduce 
the factor
\[ Y_\infty = \bigcap_{n\in \N} T^nY,
\]
which is the {\emdf invertible core} of $Y$, see \cite[Example 13.33]{EFHN}. As a consequence we have
$TY_{\infty} = Y_\infty$, and hence $T \Cex{f}{Y_\infty} = \Cex{Tf}{Y_\infty}$ for each 
$f\in X$.

\begin{lemma}\label{fac.l.HPmain-aux}
With the notation from above, $H \cap Y \subseteq Y_\infty$. 
\end{lemma}

\begin{proof}
We show $H_n \cap Y \subseteq Y_\infty$ by induction on $n\in \N_0$. Since $H_0$ consists 
of the function $\car$ only, the assertion is trivially true for $n =0$. For the step
from $n$ to $n{+}1$, suppose that $h \in H_{n+1} \cap Y$. 
Then $\Lambda h = \conj{h} Th \in H_n \cap Y$, since $Y \cap \Ell{\infty}$ is an algebra.
By induction we conclude that $\Lambda h \in Y_\infty$. Now, 
use the identity $h = \conj{\Lambda h} Th$ together with the
multiplicativity of $T$ to prove inductively
that $h \in T^mY$ for each $m \in \N$. Hence, $h\in Y_\infty$ as claimed. 
\end{proof}

We now turn to the proof of Theorem \ref{fac.t.HPmain}. 
Under the given hypothesis we shall
prove by induction on $m\in \N_0$ the assertion
\[    \forall\, h\in H_m:  h \in Y \,\,\vee\,\, h\perp Y.
\]
Note that for $m= 0$ this is trivially true. Let $m\in \N_0$, suppose that the assertion
is true for $m$ and let $h\in H_{m{+}1}$. We distinguish two cases.

\smallskip
\noindent{\bf First case.} {\em There is $n \in \N$ such that $\conj{h}T^nh \in Y$.} 

\noindent
Then, by Lemma \ref{fac.l.HPmain-aux}, 
$\conj{h}T^n h \in H \cap Y \subseteq Y_\infty$, and hence
\begin{equation}\label{fac.eq.aux}
  T^n \Cex{h}{Y} = \Cex{T^nh}{T^nY}  
= \Cex{h}{T^n Y} (\conj{h} T^n h).
\end{equation}
Since the function $\conj{h} T^n h$ has modulus equal to $\car$, it follows that
\[ \norm{\Cex{h}{Y}}_2 = \norm{T^n \Cex{h}{Y}}_2 = 
\norm{\Cex{h}{T^n Y}}_2.
\]
But $T^nY \subseteq Y$, and hence $\Cex{h}{Y} = \Cex{h}{T^nY}$. It now follows
from \eqref{fac.eq.aux} that  $\conj{h}\:\Cex{h}{Y} \in
\fix(T^n)$. Hence, by the assumption of the theorem, 
the function $h\: \conj{\Cex{h}{Y}}$ is a constant.
It follows that there is $c\in \C$ such that 
\[   \Cex{h}{Y} = c h.
\]
Since the eigenvalues of a  projection can be only $0$ and $1$, it follows that 
$c \in \{0,1\}$. If $c= 1$ then $h \in Y$; if $c=0$ then
$\Cex{h}{Y} = 0$, i.e., $h \perp Y$. This settles the first case.

\medskip

\noindent
\noindent{\bf Second case.} {\em For all $n \in \N$, $\conj{h}T^nh \notin Y$.}

\noindent
Then, since $\conj{h}T^n h \in H_m$ and by the induction hypothesis, 
$\conj{h}T^nh \perp Y$ for all $n \in \N$. Applying $T^k$ yields
\[ 0 \le k <  n \quad \Rightarrow\quad T^kh \perp T^nh \,\,\mod{T^kY}
\]
by which it is meant that $(T^kh)y \perp (T^nh) y'$ for all $y, y'\in T^k Y$. 
Now we define, for each $n \in \N_0$,
\[ f_n := h \, \conj{T^nh} \,\Cex{T^nh}{T^n Y} = h \, \conj{T^nh}\, T^n\Cex{h}{Y}.
\]
By the preceding step we have  $f_n \perp f_k$ whenever $n \neq k$. Moreover,
\begin{align*}
 \sprod{f_n}{h} & = \int  \conj{T^nh} \,\Cex{T^nh}{T^n Y} 
% = \sprod{\Cex{T^nh}{T^n Y}}{T^nh}  \\&
= \int \abs{\Cex{T^nh}{T^n Y}}^2 = \norm{f_n}^2_2
\end{align*}
since $\Cex{\cdot}{T^nY}$ is an orthogonal projection. 
This shows that $\sum_n f_n$ is the orthogonal projection of $h$ 
onto the subspace
generated by the functions $f_n$. Hence, Bessel's inequality yields
\[
1 = \norm{h}_2^2 \ge \sum_n \norm{f_n}_2^2 = \sum_n \norm{T^n \Cex{h}{Y}}_2^2
= \sum_n   \norm{\Cex{h}{Y}}_2^2.
\]
Since the sum is infinite, we must have $\Cex{h}{Y} = 0$, i.e., $h\perp Y$. 
This concludes the proof. \qed

\medskip

\begin{cor}\label{fac.l.orth}
Let $\bfX$ be a totally ergodic system. 
If $f,\:g \in G(\bfX)$ are different modulo constant functions,
then $f \perp g$. 
\end{cor}

\begin{proof}
We let $\bfY$ be  the trivial factor and 
$H$ be the smallest subset of $G = G(\bfX)$ that contains
$h := f\conj{g}$ and 
 is invariant under all the mappings $f \mapsto \conj{f}T^n f$, $n
 \in \N$.  Then the hypotheses of Theorem \ref{fac.t.HPmain} are satisfied.
It follows that either $h$ is constant or $\int h = 0$. 
\end{proof}

\begin{cor}\label{fac.c.orth}
Let $\bfX$ be a totally ergodic system, let 
$M \subseteq G(\bfX)$ be such that $G_1 M \subseteq M$. Then
\[ G(\bfX) \cap \cspan(M) = M
\]
where the closure is within $\Ell{2}(\prX)$. In particular, 
if  $n \in \N$ is such 
that $G_n(\bfX)$ is total in $\Ell{2}(\prX)$, then $G(\bfX) = G_n(\bfX)$.  
\end{cor}

\begin{proof}
For the nontrivial inclusion, 
suppose that $f\in G(\bfX) \without M$. 
Then, since $G_1$ is the set of constant functions in $G(\bfX)$
 and
$G_1 M \subseteq M$, $f$ is different modulo constants 
from every  element of $M$. By Corollary \ref{fac.l.orth},
$f \perp M$, and hence $f \notin \cspan(M)$. 

The second assertion follows from the first by letting
$M = G_n(\prX)$.  
\end{proof}

\medskip

\subsection*{The lattice of factors}\mbox{}

\medskip

With Theorem \ref{fac.t.HPmain} at hand, we can turn to the main
result of this section.
Let $\bfX$ be a totally ergodic system with quasi-discrete spectrum,
with its groups $G = G(\bfX)$ and $H= H(\bfX)$ 
of quasi-eigenvectors and quasi-eigenvalues, respectively, and its
derived homomorphism $\Lambda = \Lambda_\bfX$. 
Recall that the factors of $\bfX$
can be identified with closed and $T$-invariant sublattices 
of $X= \Ell{2}(\prX)$ containing the constants. As such, the factors
form a (complete) lattice. To every factor $\bfY$ of $\bfX$, we can
form its group $H(\bfY)$ of quasi-eigenvalues, which is in a natural way
a $\Lambda_\bfX$-invariant subgroup of $H(\bfX)$.  Indeed,
with the notational conventions from above, 
\[ G(\bfY) = Y \cap G
\]
is the group of quasi-eigenvalues of $\bfY$, and $\Lambda_\bfY = 
\Lambda\restrict{Y \cap G}$. Hence, $H(\bfY) = \Lambda(Y \cap G)$
is a $\Lambda$-invariant subgroup of $H$. 

Conversely, let $K \le H$ be any $\Lambda$-invariant subgroup of $H$.
Then 
\[ \Lambda^{-1}(K) := \{ f\in G \st \Lambda f \in K\}
\] 
is a $T$-invariant subgroup of $G$ containing $G_1$. 
Hence, $\spann(\Lambda^{-1}(K))$ is a $T$-invariant
subalgebra of $\Ell{\infty}(\prX)$
containing the constants, and therefore its closure in $\Ell{2}$,
$\cspan(\Lambda^{-1}(K))$, is a factor. 

The following theorem states that these mappings constitute 
a pair of mutually inverse order-preserving 
bijections between the lattice of factors
on one side and the lattice of $\Lambda$-invariant subgroups
on the other side.

\begin{thm}\label{fac.t.biject}
Let $\bfX$ be a totally ergodic system with quasi-discrete
spectrum, with group of quasi-eigenvectors $G= G(\bfX)$ and
derived homomorphism $\Lambda = \Lambda_\bfX$. 
Then the mappings
\[  Y  \mapsto \Lambda(Y \cap G), \qquad 
K \mapsto \cls{\spann}^{\Ell{2}}(\Lambda^{-1}(K))
\]
are mutually inverse isomorphisms between the 
lattice of factors $\bfY$ of $\bfX$
and the lattice of $\Lambda$-invariant subgroups $K$ of $H(\bfX)$.
\end{thm}

\begin{proof}
It remains to be shown that the two mappings are mutually inverse. 
Let $\bfY$ be a factor and $K := \Lambda(Y \cap G)$. Then
$\Lambda^{-1}(K) = Y \cap G$ since $G_1 \subseteq Y \cap G$. 
Denote $Y':= \cspan(Y \cap G)$. Then $Y' \subseteq Y$, and we claim
that $Y = Y'$.  

Since by Corollary \ref{fac.l.orth}   the elements of $G$ (modulo constants)
form an orthonormal basis of $X$, 
the space ${Y'}^\perp$ is generated  by those $f\in G$ such that 
$f \notin Y$. By Theorem \ref{fac.t.HPmain}, these functions also satisfy 
$f \perp Y$, so that ${Y'}^\perp \subseteq Y^\perp$. Hence, $Y \subseteq 
Y'$ as desired.

Conversely, let $K \leq H$ be any $\Lambda$-invariant subgroup and let
$Y := \cspan(\Lambda^{-1}(K))$.
Corollary \ref{fac.c.orth} applied with $M :=  \Lambda^{-1}(K)$ yields 
\[ M = G \cap  \cspan(M) = Y \cap G,
\]
from which it follows that 
\[ K = \Lambda(\Lambda^{-1}(K)) = \Lambda(M) = \Lambda(Y \cap G)
\]
as desired. 
\end{proof}

\begin{cor}\label{fac.c.fact} 
Let $\bfX$ be a totally ergodic system with quasi-discrete spectrum, and
let $\bfY$ be a factor of $\bfX$. Then $\bfY$ has quasi-discrete spectrum 
as well. 
\end{cor}

\medskip

\subsection*{Other Consequences} \mbox{}

\medskip

In the remaining part of this section, we
draw some other straightforward consequences of Theorems \ref{fac.t.HPmain} and
\ref{fac.t.biject}.

\begin{cor}
Let $\bfX$ be a totally ergodic measure-preserving system with quasi-discrete spectrum.
Then $\bfX$ has zero entropy.  
\end{cor}

\begin{proof}
As already noted, it suffices to show that every separable factor of $\bfX$ 
has zero entropy. By Corollary \ref{fac.c.fact}, such a factor has again quasi-discrete
spectrum and  by the observations from the end of the preceeding section, such 
systems have zero entropy. 

However, one can proceed differently, without making use of the results of the 
previous sections. We denote as usual $\bfX =  (\prX;\vphi)$. 
Let $\calA$ be a finite sub-$\sigma$-algebra of  $\Sigma_\prX$ and let
$\bfY$ be the factor with generating $\sigma$-algebra $\Sigma_\prY := 
\bigvee_{j=1}^\infty \vphi^{-j}\calA$. By Corollary \ref{fac.c.fact}, $\bfY$ has quasi-discrete spectrum and
hence is invertible. It follows that $\calA \subseteq \Sigma_\prY$, hence
by \cite[Cor.{}4.14.1]{Walters1982}, $h(T,\calA)=0$. 
\end{proof}

\begin{thm}\label{fac.t.disj}
Let $\bfX$ and $\bfY$ be totally ergodic factors of a system $\bfZ$,
and suppose that $\bfX$ has quasi-discrete spectrum.
Then the following assertions are equivalent:
\begin{aufzii}
\item The factor system $\bfX \wedge  \bfY$ is trivial, i.e., $\bfX \wedge \bfY = \{\car\}$.
\item The factors $\bfX$ and $\bfY$ are independent, i.e.,
$\bfX \vee \bfY \cong \bfX \times \bfY$. 
\end{aufzii}
\end{thm}

Here, $\bfX \times \bfY$ denotes the usual direct product 
of the systems $\bfX$ and $\bfY$. 

\begin{proof}
The implication (ii)$\Rightarrow$(i) is clear. For the converse, suppose that 
$X \cap Y = \C\cdot \car$. We claim that $\bfX$ and $\bfY$ satisfy
the hypotheses of Theorem \ref{fac.t.HPmain} with $H=G$ being the 
group of quasi-eigenvalues of $\bfX$.   To this end, let 
$h\in G$, $0 \neq f\in Y$ and suppose that $T^n(hf) = hf$ for some $n \ge 1$. 
Taking the modulus yields $T^n\abs{f} = \abs{f}$, and since $\bfY$ is totally
ergodic, $\abs{f}$ is constant. After rescaling we may suppose that 
$\abs{f} = \car$.  Then 
\[ \conj{h}T^nh = f \conj{T^n f} \in X \cap Y = \C\cdot \car,
\]
i.e., $T^nh = c h$ for some $\abs{c} = 1$. Since $\fact{G}{G_1}$ is torsion-free
(Lemma \ref{defs.l.torsion-free})
it follows that  $h\in G_1$, i.e., $h$ is constant. But then $T^n f = f$ and 
hence also $f$ is a constant. This establishes the claim.

Now fix again $h\in G$. 
 Then Theorem \ref{fac.t.HPmain} can be applied and yields
either $h \perp Y$ or $h\in Y$, and in the latter case it follows
by (i) that $h$ is constant. In either case
\[ \int hf = \int h \cdot \int f \quad \text{for all $f\in Y$},
\]
and since $G$ generates $X$, (ii) is proved. 
\end{proof}

The following consequence is \cite[Cor.2.4]{HahnParry1968}.

\begin{cor}\label{fac.c.disj}
Let $\bfX$ and $\bfY$ be totally ergodic systems and suppose that 
$\bfX$ has quasi-discrete spectrum. Let $T_\bfX$ and $T_\bfY$ be
 the respective Koopman operators. Then  the following assertions
are equivalent:
\begin{aufzii}
\item The systems $\bfX$ and $\bfY$ are disjoint;
\item The systems $\bfX$ and $\bfY$ have no common factors
except  the trivial one;
\item The systems $\bfX$ and $\bfY$ have no common factors with discrete
spectrum  except  the trivial one;
\item $\Pspec(T_\bfX) \cap \Pspec(T_\bfY) = \{1\}$
\end{aufzii}
\end{cor}

\begin{proof}
The implications (i)$\Rightarrow$(ii)$\Rightarrow$(iii) are straightforward
and the implication (ii)$\Rightarrow$(i) follows from Theorem \ref{fac.t.disj}. 

To see that (iii) implies (iv), let $\Gamma := 
\Pspec(T_\bfX) \cap \Pspec(T_\bfY)$, a subgroup of $\torus$.
Let $G := \Gamma^*$ be the dual group, which is compact. Then for some
$a\in G$ the rotation system $(G;a)$ is a factor of both $\bfX$ and $\bfY$.
Since this factor has discrete spectrum, by (iii) it follows
that $G= \{\car\}$, i.e., $\Gamma = \{1\}$. 

Finally, suppose that (iv) holds and that $\bfX$ and $\bfY$ have the common factor
$\bfU$. Then, by Corollary \ref{fac.c.fact}, 
$\bfU$ has quasi-discrete spectrum. 
The group of  eigenvalues of $\bfU$ is a subgroup of 
$\Pspec(T_\bfX) \cap \Pspec(T_\bfY)$, which  by (iv) is trivial. 
Hence $\bfU$ is trivial, so we have (ii).
\end{proof}

The following result appeared first in \cite[Lemma 2]{ALR2015}.

\begin{cor}\label{mqf.c.ortho}
Let  $\bfX$ be an ergodic system, $m \in \N$  and $f\in G(\bfX)$ 
such that $\Lambda_\bfX^{m} f \in \torus$ is not a root of unity. Then $f \perp \car$. 
\end{cor}

\begin{proof}
Note that $f$ cannot be a constant function.
Let $c\in \torus$ be not a root of unity, and let 
\[ H := \{ h\in G(\bfX) \st \exists m,\:n \in \N 
\,\,\text{such that}\,\, \Lambda^{m} h = c^n\car\} \cup G_1(\bfX).
\]
It is easy to see that $H$ is
$\Lambda$-invariant. Moreover,
if $h \in H$, then $h \cdot Th \cdots T^kh \in H$ for each $k\ge 1$.
We want to apply Theorem \ref{fac.t.HPmain} 
(cf. Remark \ref{fac.r.HPmain}.2) to $H$ and the
trivial factor $\bfY$. Take  $h\in H$ such that $T^k h = h$ for some 
$k \ge 1$. If there are $m,\:n \in \N$ such that $\Lambda^mh = c^n \car$,
then  $g:= \Lambda^{m{-}1}h$ is  
an eigenfunction of $T$ with eigenvalue $c^n$. It follows that
\[ c^{nk} g = T^k g = T^k \Lambda^{m{-}1}h = \Lambda^{m{-}1} T^k h =
\Lambda^{m{-}1}h = g,
\]
which implies that $nk=0$, a contradiction.  So $h \in G_1(\bfX)$, i.e.
$h$ is constant. It follows that 
Theorem \ref{fac.t.HPmain} can be applied, yielding that
all non-constant functions in 
$H$ are perpendicular to $Y = \C\car$.
\end{proof}

\section{Markov Quasi-Factors}
\label{s.mqf}

From now on we only consider separable and invertible
measure-preserving systems. The Koopman operators are usually denoted by $T$,
regardless of the system. 
Also, a  totally ergodic system with quasi-discrete spectrum
is called a {\emdf QDS-system} in the following.

\medskip

A system $\bfY$ is called 
a \textbf{Markov quasi-factor} of a system $\bfX$ 
if there is a Markov operator
$M: \Ell{1}(\prX) \to \Ell{1}(\prY)$ such that
\begin{aufzi}
\item $M$ is intertwining, i.e. $M T = T M$, and 
\item the range of $M$ is dense.
\end{aufzi}
As always with Markov operators, these properties hold
if and only if they hold for the restriction of  $M$ to the $\Ell{2}$-spaces.
Moreover, there is a dual point of view by taking adjoints: 
$\bfY$ is a Markov quasi-factor of $\bfX$ if there is an 
{\em injective} intertwining Markov operator 
$S: \Ell{1}(\prY) \to \Ell{1}(\prX)$. 

Such an operator $S$ must map eigenfunctions of $\bfY$ to eigenfunctions
of $\bfX$, resulting in $\sigma_p(T_\bfY) \subseteq \sigma_p(T_\bfX)$.
Moreover, again since $S$ is injective, the dimension of 
corresponding eigenspaces grows in passing from $\bfY$ to $\bfX$.
Hence, if $\bfX$ is (totally) ergodic, so is $\bfY$.

%It follows that Markov quasi-factors of weakly mixing systems are weakly mixing,
%and that the class of systems disjoint from weakly mixing systems
%is closed under taking Markov quasi-factors. 

Of course, if $\bfY$ is a factor of $\bfX$, then it is also a Markov
quasi-factor of $\bfX$. In general, the converse is wrong, see
\cite[Proposition 4.4]{Fraczek2010}. On the other hand, it is well known 
that a
Markov quasi-factor of an ergodic system with discrete spectrum system is a
factor. (The proof is easy: suppose that $\bfY$ is a Markov
quasi-factor of $\bfX$ where $M: \Ell{1}(\prX) \to \Ell{1}(\prY)$ is the
corresponding intertwining Markov operator. 
Let $(\ue_i)_i$ be the orthogonal basis of
eigenfunctions of $\Ell{2}(\prX)$. Then the linear span of $(M \ue_i)_i$ is
dense in $\Ell{2}(\prY)$. Furthermore, since $M$ is intertwining, $M
\ue_i \neq 0$ implies that $M \ue_i$ is an eigenfunction for all
$i$. This shows that the system $\bfY$ has discrete spectrum. Since
$\sigma_p(T_\bfY) \subseteq \sigma_p(T_\bfX)$ as well, 
$\bfY$ is a factor of $\bfX$.) It therefore has been an open question
already for some time now
whether the same is true for QDS-systems.
In this section we give an affirmative answer 
in a class of QDS-systems
with certain algebraic restrictions on the signature (Theorem
\ref{mqf.t.main}). This class includes, for example, 
the skew-shift system from Example \ref{ex.mps-qds.ss}.

\medskip

In what follows we shall employ the notion of the {\emdf derived factor} of
a QDS-system. Suppose that $\bfX$ is a
QDS-system with signature $(H, \Lambda, \eta)$. 
Then $H':=\Lambda(H) \leq H$ is a $\Lambda$-invariant
subgroup, hence Theorem \ref{fac.t.biject} yields a unique factor
$\partial \bfX$ of $\bfX$ with the signature $(H', \Lambda|_{H'},
\eta|_{H'\cap H_1})$. It is clear that if 
$\ord (H, \Lambda, \eta)$ is finite, then 
\begin{equation*}
\ord (H', \Lambda|_{H'}, \eta|_{H' \cap H_1}) = \ord (H,
\Lambda, \eta) - 1.
\end{equation*}

It has been  proved by Piekniewska 
that a Markov quasi-factor of a QDS-system is
again a QDS-system \cite[Theorem 3.1.4]{Piekniewska2013}. 
We shall show that the argument there can be refined
in order to  obtain a bound on the order of the
signature (Theorem \ref{mqf.t.ordbd} below).  The proof, which 
is merely a closer inspection of the original one,  rests on
the following two powerful  results from the literature.

\begin{thm}\cite[Proposition 5.1]{Fraczek2010}.
\label{mqf.thm.frlem}
If $\bfY$ is a Markov quasi-factor of an ergodic system
$\bfX$, then $\bfY$ is a factor
of some infinite ergodic self-joining of $\bfX$.
\end{thm}

\begin{thm} \cite[Th\'{e}or\`{e}me 4]{Lesigne1993}. 
\label{mqf.thm.les}
Let $\bfX$ be a totally ergodic system with
group of quasi-eigenfunctions $G(\bfX)$. 
Then for 
every $k \geq 0$ and every $f\in \Ell{2}(\prX)$ 
the following assertions
 are equivalent:
\begin{aufzii}
\item $f \in G_{k+1}(\bfX)^{\perp}$;
\item For a.e. $x \in X$ and
for each $P \in \R_k[t]$ and each 
continuous periodic function $\chi$ on $\R$ one has
\begin{equation*}
\lim\limits_{N \to \infty} \frac 1 N \sum\limits_{n=0}^{N-1}
\chi(P(n)) (T^n f)(x) = 0.
\end{equation*}
\end{aufzii}
\end{thm}

Here, $\R_{k}[t]$ 
denotes the space of all real polynomials of one variable 
of degree less or equal to $k$. 

\begin{proof}
The case $k\ge 1$ is treated in Lesigne's paper
\cite{Lesigne1993}. (Recall from Remark 
\ref{defs.r.labelling} that  
$G_{k{+}1}(\bfX) = E_k(\bfX)$ for $k\ge 1$
in Lesigne's terminology.) The case $k=0$
holds by Birkhoff's ergodic theorem. 
\end{proof}

It was observed in \cite{Piekniewska2013} 
that the total ergodicity of $\bfX$ is not
required for the proof of the implication 
$\mathrm{(ii)} \Rightarrow \mathrm{(i)}$.

\medskip

We can now state
and prove the announced refinement of Piekniewska's result.
Its proof is completely along the lines of her original argument.

\begin{thm}
\label{mqf.t.ordbd}
Let $\bfY$ be a Markov quasi-factor of a QDS-system $\bfX$.  
Then $\bfY$ is again a QDS-system and
\begin{equation*}
\ord(H({\bfY}),\Lambda_{\bfY}, \eta_{\bfY}) \leq 
\ord(H({\bfX}), \Lambda_{\bfX}, \eta_{\bfX}).
\end{equation*}
\end{thm}

\begin{proof}
As $\bfY$ is a  Markov quasi-factor of $\bfX$ and $\bfX$ is totally ergodic,
$\bfY$ is totally ergodic and
a factor of an ergodic (countably) infinite self-joining $\bfZ$, say, 
of $\bfX$ (Theorem \ref{mqf.thm.frlem}). 
In this situation 
we may consider the different $T$-intertwining embeddings $J_n : \Ell{2}(\bfX)
\to \Ell{2}(\bfZ)$ which generate the joining $\bfZ$. 
For $f_1, \dots, f_m \in \Ell{\infty}(\prX)$ we abbreviate
\[ f_1 \tensor \dots \tensor f_m := (J_1f_1) \cdot (J_2f_2)
 \cdots (J_mf_m)
\in \Ell{\infty}(\prZ).
\]
It is then clear that if $k \in \N_0$ and
each $f_j \in G_{k}(\bfX)$, then 
$f_1 \tensor \dots \tensor f_m \in G_{k}(\bfZ)$.

Further,  $\bfX$ is a QDS-system, $G(\bfX)$ is a total 
subset of $\Ell{2}(\bfX)$. As $\bfZ$ is an infinite self-joining of
$\bfX$, the elements of the form $f_1 \tensor \dots \tensor f_m$ with
each $f_m \in G(\bfX)$ form a total 
subset of  $\Ell{2}(\bfZ)$. 
In particular, $G(\bfZ)$ is a total subset of  $\Ell{2}(\bfZ)$.

Let now $k \in \N_0$ and
suppose that $f \in \Ell{2}(\prY)$ is such that $f \perp G_{k{+}1}(\bfY)$.
Then, by Theorem \ref{mqf.thm.les}, for
a.e. $y \in Y$, for each $P \in \R_k[t]$ and each continuous 
periodic function $\chi$ on $\R$ we have that
\begin{equation}
\label{eq.lesigne}
\lim\limits_{N \to \infty} \frac 1 N \sum\limits_{n=0}^{N-1} \chi(P(n)) T^nf(y) = 0.
\end{equation}
Identifying $f$ with 
an element in $\Ell{2}(\prZ)$ we see that
we may start the assertion with ``for almost every $y\in Z$'' here. 
Since the second implication of Theorem \ref{mqf.thm.les} 
does only require ergodicity, 
we conclude that $f \perp G_{k{+}1}(\bfZ)$.

Consequently, if  $f \perp G(\bfY)$, 
then $f\perp G(\bfZ)$, which implies that $f =0$. This shows that
$\bfY$ has quasi-discrete spectrum. Now suppose in addition 
that $k = \ord(H({\bfX}), \Lambda_{\bfX}, \eta_{\bfX})$ is finite. 
Then $G_k(\bfX)$ is total in $\Ell{2}(\prX)$ and 
hence $G_k(\bfZ)$ is total $\Ell{2}(\prZ)$. 
As above, it follows that $G_k(\bfY)$ is 
total in $\Ell{2}(\prY)$.

In particular, if $f\in G_{k{+}1}(\bfY) \setminus G_k(\bfY)$, then 
$f\perp G_k(\bfY)$, which implies that $f= 0$. This is impossible,
so $G_{k{+}1}(\bfY) \setminus G_k(\bfY) = \emptyset$. And this was the claim.
\end{proof}

For the main theorem below it will be important 
to know that the derived factors are `respected' 
under Markov quasi-factor maps of QDS-systems of order $2$. 
This is our next step.

\begin{thm} %[Derivation stability]
\label{mqf.t.derstab}
Let $\bfX$ be a QDS-system with
signature $(H(\bfX),\Lambda_{\bfX},\eta_{\bfX})$ of order $2$. 
Let $\bfY$ be a Markov quasi-factor of $\bfX$. 
Then $\partial \bfY$ is a factor of $\partial \bfX$. In other words:
\begin{equation}
\label{eq.stab1}
 \Lambda_\bfY(H_2(\bfY)) \subseteq \Lambda_\bfX(H_2(\bfX))
\end{equation}
when viewed as subgroups of $\torus$. 
\end{thm}

Note that since $\bfX$ has order $2$, $\bfY$ is a QDS-system of order
at most $2$ (Theorem \ref{mqf.t.ordbd}). 
Hence $\partial \bfY$ is a QDS-system of order at most $1$,
i.e., a system with discrete spectrum. From the spectral considerations
it follows that 
$H_1(\bfY) \subseteq H_1(\bfX)$ as subsets of $\T$.

\begin{proof}
Let, as in the proof of Theorem \ref{mqf.t.ordbd}, 
be $\bfZ$ an infinite ergodic self-joining
of $\bfX$ having $\bfY$ as a factor. 
By definition of the groups $H_2(\bfX)$ and $H_2(\bfY)$, 
\eqref{eq.stab1} is the same as
\begin{equation}
\label{eq.stab}
\Lambda_{\bfY}^2(G_3(\bfY)) \subseteq \Lambda_{\bfX}^2(G_3(\bfX)).
\end{equation}
In order to prove this, let $f \in G_3(\bfY)$ be such that
\begin{equation*}
\Lambda_{\bfZ}^2 f = c \in \Lambda_{\bfY}^2(G_3(\bfY)) \setminus \Lambda_{\bfX}^2(G_3(\bfX)).
\end{equation*}
Clearly, $c \in H_1(\bfY) \subseteq H_1(\bfX)$ is irrational. Let 
$f_1,\dots,f_m \in G_3(\bfX)$ be arbitrary. Then
\begin{equation*}
\Lambda_{\bfZ}^2(\conj{f} \cdot (f_1 \otimes \dots \otimes f_m)) = \conj{c}
 \cdot \Lambda_{\bfX}^2(f_1) \cdot \dots \cdot \Lambda_{\bfX}^2(f_m) \in 
H_1(\bfX) \cap \conj{c}\Lambda_{\bfX}^2(G_3(\bfX)).
\end{equation*}
As an element of $H_1(\bfX)$,  
it is either irrational or is equal to $1$. But, in fact, 
it cannot be equal to $1$ because of the assumption 
$c \in \Lambda_{\bfY}^2(G_3(\bfY)) \setminus \Lambda_{\bfX}^2(G_3(\bfX))$. 
We conclude by Corollary 
\ref{mqf.c.ortho} that $f$ is orthogonal to all such tensors. 
But then $f = 0$ (by the density of the span of the set of all tensors), a 
contradiction.
\end{proof}

Next, we recall some basic algebraic results. First of all, we state
the following lemma. The proof can be found in \cite[Lemma 7.2]{Lang2002}.

\begin{lemma}
\label{mqf.l.dirsumfr}
Let $f: A \to A'$ be a surjective homomorphism of Abelian groups, and
assume that $A'$ is free. Let $B$ be the kernel of $f$. Then there
exists a subgroup $C$ of $A$ such that the restriction of $f$ to $C$
induces an isomorphism of $C$ with $A'$, and such that $A = B
\oplus C$.
\end{lemma}

Using Lemma \ref{mqf.l.dirsumfr} and the fact that a subgroup of a
free Abelian group is a free Abelian group as well (see \cite[Theorem
7.3]{Lang2002}) one can easily prove the following lemma.
\begin{lemma}
\label{mqf.l.dirsum}
Let $H$ be an Abelian group and let $\pi: H \to H'$ be a
homomorphism such that  the following assumptions hold:
\begin{aufzi}
\item $H'$ is a free Abelian group; 

\item $\ker \pi \leq H$ is a free Abelian group.
\end{aufzi}
Then there is a subgroup $K \leq H$ isomorphic via $\pi|_K$ to the subgroup
$\ran \pi \leq H'$ such that $H = K \oplus \ker \pi$. Any such $K$ is
a free Abelian group, and the group $H$ is free Abelian as well.
\end{lemma}

Finally, we arrive at the main theorem of this section. 

\begin{thm}
\label{mqf.t.main}
Let $\bfX$ 
be a QDS-system with signature $(H(\bfX),\Lambda_{\bfX},\eta_{\bfX})$ 
of order at most $2$ such that the group 
$H_1(\bfX)$ of eigenvalues is a free Abelian group. 
Then each Markov quasi-factor of $\bfX$ is a factor of $\bfX$.
\end{thm}

\begin{proof}
The system $\bfX$ is a QDS-system with 
signature $(H(\bfX),\Lambda_{\bfX},\eta_{\bfX})$  
of order at most $2$. Since $\bfY$ is a
Markov quasi-factor of $\bfX$,  $\bfY$ is also a QDS-system
with the signature $(H(\bfY),\Lambda_{\bfY},\eta_{\bfY})$ of
order at most $2$.

Our goal is to define the group homomorphisms 
$\alpha_1,\alpha_2$ such that 
$\alpha_2|_{H_1(\bfY)} = \alpha_1$, 
$\eta_{\bfX} \circ \alpha_1 = \eta_{\bfY}$ 
and such that  the diagram
\begin{equation*}
\xymatrix{ \mathbf 1  & H_1(\bfX)  \ar[l] & H_2(\bfX) \ar[l]^{\Lambda_{\bfX}} \\
               \mathbf 1 \ar[u]^{\id}  & H_1(\bfY) \ar[u]_{\alpha_1}  \ar[l] &  H_2(\bfY) \ar[u]_{\alpha_2} \ar[l]^{\Lambda_{\bfY}}}
\end{equation*}
is commutative. 
Then, by Theorem \ref{fac.t.biject}, the statement of the theorem
follows. 

As $\bfY$ is a Markov quasi-factor of $\bfX$ one has
a natural inclusion $H_1(\bfY) \subseteq H_1(\bfX)$, and
we choose $\alpha_1$ to be this inclusion map. Then 
clearly $\eta_{\bfX} \circ \alpha_1 = \eta_{\bfY}$ as $\eta_\bfX$ and
$\eta_\bfY$ just map constant functions to their values. 

In order to define the homomorphism $\alpha_2$,  observe that $\ker
\Lambda_{\bfY} = H_1(\bfY) \subseteq H_2(\bfY)$. 
Fix a decomposition $H_2(\bfY) =
H_1(\bfY) \oplus K$ for some free Abelian subgroup 
$K \leq H_2(\bfY)$, 
given by Lemma \ref{mqf.l.dirsum}. We let $\alpha_2|_{H_1(\bfY)}:=\alpha_1$.

Suppose that $\{ \varepsilon_j \}_{j \in I}$ is a basis for $K$. 
Since $ \Lambda_\bfY(H_2(\bfY)) \subseteq \Lambda_\bfX(H_2(\bfX))$
by Theorem \ref{mqf.t.derstab},  for
every basis element $\varepsilon_j$, 
there is $\delta_j \in H_2(\bfX)$ such that
\begin{equation*}
\alpha_1 \Lambda_{\bfY}(\varepsilon_j) = \Lambda_{\bfX} (\delta_j).
\end{equation*}
Defining $\alpha_2$ by $\alpha_2(\varepsilon_j):=\delta_j$ 
for every $j \in I$ completes the proof.
\end{proof}

\bigskip

{\bf Acknowledgements.}\
The authors would like to thank M.~Lema{\'n}czyk (Toru{\'n}) for his
long-term sympathetic support and  his kind hospitality. In particular,
the second author is grateful for the two generous invitations
to Toru\'{n} in November 2014 and March 2016, 
which led to the results of Section 6.  The second author
  also kindly acknowledges the financial support from
Delft Institute of Applied Mathematics.

\bibliographystyle{acm} %plain, acm, apalike, unsrt?

%\bibliography{quasidiscretebib}

\begin{thebibliography}{HvN42}

\bibitem[ELD15]{ALR2015}
{\sc {El Abdalaoui}, E.H., {Lema{\'n}czyk}, M., and {De La Rue}, T.}
\newblock Automorphisms with quasi-discrete spectrum, multiplicative functions and average orthogonality along short intervals.
\newblock {\em ArXiv e-Print} (2015), arXiv:1507.04132. 

\bibitem[Abr62]{Abramov1962}
{\sc Abramov, L.~M.}
\newblock Metric automorphisms with quasi-discrete spectrum.
\newblock {\em Izv. Akad. Nauk SSSR Ser. Mat. 26\/} (1962), 513--530.


\bibitem[Bro76]{Brown1976}
{\sc Brown, J.~R.}
\newblock {\em Ergodic theory and topological dynamics}.
\newblock Academic Press [Harcourt Brace Jovanovich, Publishers], New
  York-London, 1976.
\newblock Pure and Applied Mathematics, No. 70.


\bibitem[EFHN]{EFHN}
{\sc Eisner, T., Farkas, B., Haase, M., and Nagel, R.}
\newblock {\em Operator {T}heoretic {A}spects of {E}rgodic {T}heory.}
\newblock Graduate Texts in Mathematics 272, Springer, 2015.


\bibitem[Ell69]{Ellis1969}
{\sc Ellis, R.}
\newblock {\em Lectures on topological dynamics}.
\newblock W. A. Benjamin, Inc., New York, 1969.




\bibitem[Fra10]{Fraczek2010}
{\sc Fr{\k{a}}czek, K. and Lema{\'n}czyk, M.}
\newblock A note on quasi-similarity of {K}oopman operators.
\newblock {\em J. Lond. Math. Soc. (2)} 82 (2010), 361--375.

\bibitem[HP65]{HahnParry1965}
{\sc Hahn, F., and Parry, W.}
\newblock Minimal dynamical systems with quasi-discrete spectrum.
\newblock {\em J. London Math. Soc. 40\/} (1965), 309--323.

\bibitem[HP68]{HahnParry1968}
{\sc Hahn, F., and Parry, W.}
\newblock Some characteristic properties of dynamical systems with
  quasi-discrete spectra.
\newblock {\em Math. Systems Theory 2\/} (1968), 179--190.

\bibitem[HvN42]{HalmosVonNeumann1942}
{\sc Halmos, P.~R., and von Neumann, J.}
\newblock Operator methods in classical mechanics. {II}.
\newblock {\em Ann. of Math. (2) 43\/} (1942), 332--350.

\bibitem[Lan02]{Lang2002}
{\sc Lang, S.}
\newblock {\em Algebra.}
\newblock Graduate Texts in Mathematics 211, Springer-Verlag, New York, 2002.


\bibitem[Les93]{Lesigne1993}
{\sc Lesigne, E.}
\newblock Spectre quasi-discret et th\'eor\`eme ergodique de
              {W}iener-{W}intner pour les polyn\^omes.
\newblock {\em Ergodic Theory Dynam. Systems} 13 (1993), 767--784.


\bibitem[Par68]{Parry1968}
{\sc Parry, W.}
\newblock Zero entropy of distal and related transformations.
\newblock In {\em Topological {D}ynamics ({S}ymposium, {C}olorado {S}tate
  {U}niv., {F}t. {C}ollins, {C}olo., 1967)}. Benjamin, New York, 1968,
  pp.~383--389.


\bibitem[Par81]{Parry1981}
{\sc Parry, W.}
\newblock {\em Topics in ergodic theory}, vol.~75 of {\em Cambridge Tracts in
  Mathematics}.
\newblock Cambridge University Press, Cambridge-New York, 1981.

\bibitem[Pie13]{Piekniewska2013}
{\sc Pi\k{e}kniewska, A.}
\newblock Strong regularity of affine cocycles over irrational
rotations.
\newblock {\em PhD Thesis, Toru\'n} (2013).


\bibitem[Wal82]{Walters1982}
{\sc Walters, P.}
\newblock {\em An Introduction to Ergodic Theory}, vol.~79 of {\em Graduate
  Texts in Mathematics}.
\newblock Springer-Verlag, New York-Berlin, 1982.


\end{thebibliography}

\end{document}